\documentclass[11pt]{article}
\usepackage[utf8]{inputenc}
\usepackage[english]{babel}
\usepackage{amsthm}
\usepackage{amsmath}
\usepackage{amsfonts}
\usepackage{amssymb}
\usepackage{interval}
\usepackage{relsize}
\usepackage{enumitem}
\usepackage{tabularx}
\usepackage{xcolor}
\usepackage{fullpage}
\usepackage[numbers]{natbib}
\usepackage[hang,flushmargin]{footmisc}
\newcommand{\Tau}{\mathcal{T}}
\AtBeginDocument{\colorlet{defaultcolor}{.}}

\title{Adaptable and conflict colouring multigraphs with no cycles of length three or four}

\author{Jurgen Aliaj \and Michael Molloy}
\date{
    {\small
        University of Toronto \\ [1ex]
        Department of Computer Science \\ [1ex]
        10 King's College Road \\ [1ex]
        Toronto, ON, Canada, M5S 3G4 \\ [1ex]
        \texttt{\{aliaj,molloy\}@cs.toronto.edu} \\ [3ex]
    }
    \today
}

\newtheorem{thm}{Theorem}[section]
\newtheorem{cor}[thm]{Corollary}
\newtheorem{fact}[thm]{Fact}

\theoremstyle{definition}
\newtheorem{definition}[thm]{Definition}

\theoremstyle{definition}
\newtheorem{example}[thm]{Example}
\newtheorem{modification}{Modification}
\newtheorem{observation}{Observation}

\theoremstyle{plain}
\newtheorem{lem}{Lemma}[section]
\newtheorem{claim}[lem]{Claim}

\theoremstyle{definition}
\newtheorem{remark}[lem]{Remark}

\providecommand{\keywords}[1]{\textbf{Keywords:} #1}

\begin{document}
\maketitle

\begin{abstract}
    The \textcolor{defaultcolor}{\textit{adaptable choosability}} of a multigraph $G$, denoted $\mathrm{ch}_a(G)$, is the smallest integer $k$ such that any edge labelling, $\tau$, of $G$ and any assignment of lists of size $k$ to the vertices of $G$ permits a list colouring, $\sigma$, of $G$ such that there is no edge $e = uv$ where $\tau(e) = \sigma(u) = \sigma(v)$. Here we show that for a multigraph $G$ with maximum degree $\Delta$ and no cycles of length 3 or 4, $\mathrm{ch}_a(G) \leq (2\sqrt{2}+o(1))\sqrt{\Delta/\ln\Delta}$. Under natural restrictions we can show that the same bound holds for the \textcolor{defaultcolor}{\textit{conflict choosability}} of $G$, which is a closely related parameter recently defined by Dvo{\v r}{\'a}k, Esperet, Kang and Ozeki. \\ [1em]
    \noindent
\end{abstract}

\section{Introduction}
Hell and Zhu introduced the adaptable chromatic number \cite{hell-zhu-adaptable}. This variant on the chromatic number was also introduced independently (under different names) by Erd{\"o}s and Gy{\'a}rf{\'a}s \cite{split}, Cochand and Duchet \cite{emulsive}, and Archer \cite{chromcapacity}.

The basic idea is this: we are given an edge labelling of a graph $G$ and we assign colours to the vertices of $G$ such that there is no edge $e=uv$ where $u$, $v$, and $e$ are all the same colour. Such a colouring is called an \textit{adaptable colouring} of $G$. I.e. in conventional graph colouring each edge forbids its endpoints from both receiving the same colour, whereas in adaptable colouring each edge forbids its endpoints from both receiving a particular colour, namely its edge label. Note that an adaptable colouring may be improper in the sense that two adjacent vertices may both receive the colour red, so long as the edge connecting them is not also labelled with the colour red.

We can take this a step further by including multiple edges between adjacent vertices, to further constrain the colours assigned to them. In fact, a conventional colouring problem with $k$ colours can be reduced to an adaptable colouring problem by passing to the corresponding multigraph with multiplicity $k$, and such that each edge between two endpoints is given a unique colour from $[k] = \{1,\dots,k\}$. Therefore it is natural---and in fact important---to allow multiple edges in our graphs.

For a formal definition: let $G = (V,E)$ be a multigraph. Let $\tau: E \to \mathbb{N}$ be an edge labelling of $G$, we say that a vertex colouring, $\sigma: V \to \mathbb{N}$, of $G$ is a $\textit{proper adaptable colouring}$ if there is no edge $e = uv$ such that $\sigma(u) = \sigma(v) = \tau(e)$. The $\textit{adaptable chromatic number}$ of $G$, denoted $\chi_a(G)$, is the minimum integer $k$ such that every edge labelling of $G$ using labels from $[k]$ permits a proper adaptable colouring of $G$ from $[k]$. Since every proper vertex colouring of $G$ is a proper adaptable colouring of $G$ for any edge labelling, we can see that $\chi_a(G) \leq \chi(G)$, where $\chi(G)$ is the usual $\textit{chromatic number}$ of $G$. Applications of this problem include job scheduling \cite{job-sched}, matrix and list partitions of graphs \cite{matrix-part, list-part}, as well as full constraint satisfaction problems \cite{csp}.

Here we present a result for adaptable list colouring, which is defined naturally. Suppose each $v \in G$ is assigned a list $L(v) \subset \mathbb{N}$ of candidate colours. Then a $\textit{proper adaptable list colouring}$ of $G$ is a proper adaptable colouring where each vertex $v$ is assigned a colour from $L(v)$. The \textcolor{defaultcolor}{$\textit{adaptable choosability}$} of $G$, denoted $\mathrm{ch}_a(G)$, is the smallest integer $k$ such that any edge labelling of $G$ and any assignment of lists of size $k$ to the vertices of $G$ permits a proper adaptable list colouring of $G$. If $\mathrm{ch}_a(G) = k$, then assigning $L(v) = [k]$ for each $v \in G$ permits a proper adaptable colouring, hence $\chi_a(G) \leq \mathrm{ch}_a(G)$. Moreover, since any proper vertex colouring of $G$ from lists of size $k$ permits a proper adaptable colouring of $G$ for any edge labelling, we have $\mathrm{ch}_a(G) \leq \mathrm{ch}(G)$, where $\mathrm{ch}(G)$ is the usual \textcolor{defaultcolor}{$\textit{choosability}$} of $G$. Adaptable colouring and adaptable list colouring have been studied in \textcolor{defaultcolor}{\cite{adaptablechrom-graphprod, molloy-thron-2, hell-zhu-adaptable, adaptablechoosability2009,adaptablechoosabilityplanarcycle, adaptedplanar}}.

In the case of conventional graph colouring, much attention has been given to colouring graphs of high girth \cite{trianglefree-dabrowski, trianglefree-thomassen, randomgirth5}, as typically fewer colours are required. We will see that the same phenomenon can be observed with adaptable list colouring.

Two results in particular are of interest to us. The first is by Kim \cite{kim95} who showed that if $\Delta$ is the maximum degree of a simple graph $G$ with girth at least 5 (i.e. no cycles of length 4 or less), then $\mathrm{ch}(G) \leq (1 + o(1))\frac{\Delta}{\ln\Delta}$. The second is by Molloy and Thron \cite{molloy-thron}, who showed that for a general multigraph $G$ with maximum degree $\Delta$, $\mathrm{ch_a}(G) \leq (1 + o(1))\sqrt{\Delta}$. In a sense, we combine these results by presenting an upper bound on the adaptable \textcolor{defaultcolor}{choosability}, which we show for essentially the same class of graphs considered by Kim.

Recall that it is important to allow multiple edges in the graphs we consider. So if we would like to study adaptable colouring in a high girth setting, we must define a notion of high girth for multigraphs. The most natural course of action is to permit 2-cycles, i.e. multiple edges, while disallowing other short cycles in our graphs. In particular we consider multigraphs with no cycles of length 3 or 4, which is the most natural analogue to Kim's setting. We get the following result:

\begin{thm}\label{adaptable-theorem}
Let $G$ be a multigraph with no cycles of length 3 or 4 and maximum degree $\Delta$, then $\mathrm{ch}_a(G) \leq (2\sqrt{2} + o(1))\sqrt{\Delta/\ln\Delta}$.
\end{thm}

We comment briefly on the tightness of \textcolor{defaultcolor}{Theorem \ref{adaptable-theorem}}. For any $g$ and $\Delta \geq 3$, random $\Delta$-regular graphs have girth at least $g$ and satisfy $\chi(G) \geq (\frac{1}{2} + o(1))\frac{\Delta}{\ln\Delta}$ with high probability \cite{frieze1992independence}. Moreover it is known that for any graph $G$, $\chi_a(G) \geq (1 + o(1))\sqrt{\chi(G)}$ by \cite{adap-versus-chrom}. Combining these two facts implies the existence of a graph $G$ with maximum degree $\Delta$ and no 3 or 4-cycles, satisfying $\chi_a(G) \geq (\frac{\sqrt{2}}{2} + o(1))\sqrt{\Delta/\ln\Delta}$. And since $\mathrm{ch}_a(G) \geq \chi_a(G)$, this implies the bound in \textcolor{defaultcolor}{Theorem \ref{adaptable-theorem}} is within a factor of 4 of the best possible (note that Kim's result is within a factor of 2 of the best possible, see remark \ref{remark1} to understand why this changes to a 4 in our case).

We achieve \textcolor{defaultcolor}{Theorem \ref{adaptable-theorem}} by showing a more general result for a problem called conflict colouring. In adaptable colouring, we are given an edge labelling $\tau$ so that each edge $e = uv$ constrains the colours assigned to $u$ and $v$ symmetrically. I.e. the edge colour $c = \tau(e)$ acts as a constraint preventing $u$ and $v$ from both receiving the same colour $c$. In conflict colouring we allow for the possibility of an asymmetric constraint, meaning that instead of the edge label being a single colour, it can now be an ordered pair of colours $(c, c')$ which is meant to prevent $u$ from receiving $c$ and $v$ from receiving $c'$ simultaneously. The problem was introduced independently by Dvo{\v r}{\'a}k and Postle \cite{correspond-colour} and Fraigniaud, Heinrich and Kosowski \cite{local-con-colour}.

We should point out that in \cite{correspond-colour, local-con-colour}, the authors present slightly different (but equivalent) representations of what we call the conflict colouring problem. For instance, in \cite{correspond-colour} the authors consider \textit{correspondence colouring}. In this variant we are given a simple graph where each edge $e=uv$ corresponds to a partial matching, $M_e$, between the candidate colours $L(u)$ and $L(v)$. A pair $f \in M_e$ forbids $u$ and $v$ from both receiving the colours corresponding to the entries of $f$. This is of course the same problem we consider here, the only important difference being that in our case the pairs in $M_e$ are instead included in our graph as parallel edges between $u$ and $v$, and therefore they contribute to the maximum degree. Our setting most closely resembles that of Dvo{\v r}{\'a}k, Esperet, Kang and Ozeki \cite{least-con-choose}, and we use the same notion of conflict degree (we define this in a moment) seen in \cite{local-con-colour}.

For a multigraph $G = (V,E)$, Dvo{\v r}{\'a}k et al. \cite{least-con-choose} defined the $\textit{conflict choosability}$ of $G$, denoted $\mathrm{ch}_{\nleftrightarrow}(G)$, as the smallest integer $k$ such that any edge labelling $\tau: E \to \mathbb{N} \times \mathbb{N}$ and any assignment of lists $L(v) \subset \mathbb{N}$ to every $v \in G$, each of size $k$, permits a vertex colouring $\sigma: V \to L(v)$ such that there is no edge $e = uv \in G$ where $\tau(e) = (\sigma(u), \sigma(v))$. We call $\sigma$ a $\textit{proper conflict colouring}$.

The reader may notice some ambiguity here, since for a label $\tau(e) = (c,c')$ it is not clear which of the endpoints of $e$ should be forbidden from receiving $c$ and which should be forbidden from receiving $c'$. We will define some convenient notation to overcome this (see the comment at the beginning of section 3). For now, what we mean will be made clear by the context.

Observe that for any edge labelling and list assignments of size $k$, we can rename the colours in $L(v)$ as well as the colours assigned by $\tau$ accordingly while preserving the existence of a proper conflict colouring. So in fact we may assume $L(v) = [k]$ for each $v \in G$. I.e. conflict colouring and conflict list colouring are one and the same.

It turns out that it is not possible to obtain the same bound for conflict choosability as in \textcolor{defaultcolor}{Theorem \ref{adaptable-theorem}}, which will be made clear by the following example.

\begin{example}\label{counterexample1}
Let $\Delta$ be a natural number and $L$ be a set of colours with size $\ell \leq \sqrt{\Delta}$. Consider the multigraph $G$ with only two vertices, $u$ and $v$, each assigned the same set $L$ of colours and connected by $\ell^2$ edges. Note the maximum degree of $G$ is $\ell^2 \leq \Delta$, so if needed we can increase the maximum degree to be exactly $\Delta$ by adding dummy vertices. Define $\tau$ such that each edge between $u$ and $v$ is given a unique label from one of the $\ell^2$ possible pairs of colours, and observe that $G$ does not have a proper conflict colouring. I.e. $G$ is a multigraph with no cycles of length 3 or 4, but $\mathrm{ch}_{\nleftrightarrow}(G) > \sqrt{\Delta}$.
\end{example}

To investigate what went wrong in the previous example we need to first define a parameter called the conflict degree.

\begin{definition}\label{conflict-degree-definition} Given a multigraph $G=(V,E)$, an edge labelling $\tau: E \to \mathbb{N} \times \mathbb{N}$, a colour $c$, and a pair of vertices $u$ and $v$, define $d(\tau, c, u, v)$ to be the number of edges $e=uv \in G$ such that $\tau(e) = (c, c')$ for some colour $c'$. Define the \textit{conflict degree} of $\tau$, denoted $D(\tau)$, to be the maximum of $d(\tau, c, u, v)$ over all colours $c$ and pairs $u, v$.
\end{definition}

Notice that in \textcolor{defaultcolor}{Example \ref{counterexample1}} we considered an edge labelling $\tau$ with $D(\tau) = \ell$, the size of our lists. In other words if we allow the conflict degree to become too high then we have no hope of finding a proper conflict colouring. In the specific case where we consider lists of size $\ell = (2\sqrt{2} +o(1))\sqrt{\Delta/\ln\Delta}$, the same problem occurs if $D(\tau)$ is as high as $\ell^{15/16}$, which will be made clear by \textcolor{defaultcolor}{Example \ref{counterexample2}}. Note that \textcolor{defaultcolor}{Example \ref{counterexample2}} is more involved, and so we include the details at the end of the section. We simply summarize the main point here.

\begin{fact}\label{badfact}
For any constant $\alpha \geq 1$, there is a multigraph $G$ with maximum degree $\Delta$, no cycles of length 3 or 4, an assignment of lists each with size $\ell=\alpha\sqrt{\Delta/\ln\Delta}$, an edge labelling $\tau$ with $D(\tau) = \ell^{f(\alpha)}$, and such that $G$ does not have a proper conflict colouring, \textcolor{defaultcolor}{where $f(\alpha)$ is non-negative for $\alpha \geq 1$, non-decreasing, and tending to 1 as $\alpha$ approaches infinity. In particular, when $\alpha = 2\sqrt{2} + o(1)$ we have $f(\alpha) = \frac{15}{16}$}.
\end{fact}
 
The function $f$ will be defined precisely in Example \ref{counterexample2}.

Clearly, \textcolor{defaultcolor}{Fact \ref{badfact}} requires us to restrict the types of edge labellings we consider. Moving forward, we restrict the type of edge labelling that is allowed on our graph by imposing an upper bound on the conflict degree. Such an approach has been taken in \cite{local-con-colour}. Note that since we consider lists of size $\ell = (2\sqrt{2} + o(1))\sqrt{\Delta/\ln\Delta}$, \textcolor{defaultcolor}{Fact \ref{badfact}} means we must restrict the conflict degree to be at most $\ell^{15/16} \approx \Delta^{15/32}$. Unfortunately we run into problems even when allowing the conflict degree to be this large. Instead we aim for something closer to $\Delta^{1/4}$, we make no attempt to determine if this is optimal. We arrive at the following \textcolor{defaultcolor}{generalization} of \textcolor{defaultcolor}{Theorem \ref{adaptable-theorem}}:

\begin{thm}\label{conflict-theorem}
For any $\epsilon > 0$, there exists $\Delta_0$ such that if $G$ is a multigraph with maximum degree $\Delta > \Delta_0$, no cycles of length 3 or 4, an assignment of lists $L(v)$ for each vertex $v$ such that $|L(v)| \geq (2\sqrt{2}+\epsilon)\sqrt{\Delta/\ln\Delta}$, and an edge labelling $\tau$ with $D(\tau) \leq \Delta^{\frac{1}{4}\epsilon^2/(\epsilon + 5)^2}$, then there exists a proper conflict colouring of $G$ from these lists.
\end{thm}

The important point here is that there is a tradeoff between the bound on our conflict degree and our list sizes. We can make our list sizes arbitrarily close to $2\sqrt{2}\sqrt{\Delta/\ln\Delta}$ and the conflict degree arbitrarily close to $\Delta^{1/4}$, but not at the same time.

It is easy to see that \textcolor{defaultcolor}{Theorem \ref{conflict-theorem}} implies \textcolor{defaultcolor}{Theorem} \ref{adaptable-theorem}. The key observation is that an instance of an adaptable colouring problem is exactly an instance of a conflict colouring problem with conflict degree 1.

Notice that since \textcolor{defaultcolor}{Theorem \ref{adaptable-theorem}} is implied by \textcolor{defaultcolor}{Theorem \ref{conflict-theorem}}, this means the discussion following the statement of \textcolor{defaultcolor}{Theorem \ref{adaptable-theorem}} applies to \textcolor{defaultcolor}{Theorem \ref{conflict-theorem}} as well. That is, the choice of list sizes in \ref{conflict-theorem} is within a factor of 4 of the best possible, even when $D(\tau) \leq 1$.

As mentioned, we cannot quite allow our conflict degree to be as large as $\Delta^{1/4}$, but by allowing the list size to be large we can come close. A direct result of \textcolor{defaultcolor}{Theorem \ref{conflict-theorem}} is the following more visually appealing statement, which can be seen by taking $\epsilon=45$.

\begin{cor}\label{clean-theorem}
There exists $\Delta_0$ such that if $G$ is a multigraph with maximum degree $\Delta > \Delta_0$, no cycles of length 3 or 4, an assignment of lists $L(v)$ for each vertex $v$ such that $|L(v)| \geq 50\sqrt{\Delta/\ln\Delta}$, and an edge labelling $\tau$ with $D(\tau) \leq \Delta^{1/5}$, then there exists a proper conflict colouring of $G$ from these lists.
\end{cor}

Actually, if one is more careful with the analysis (in particular the proof of \textcolor{defaultcolor}{Lemma \ref{listsizelargelem}}), then \textcolor{defaultcolor}{Theorem \ref{conflict-theorem}} can be strengthened slightly and as a result the ``50" in \textcolor{defaultcolor}{Corollary \ref{clean-theorem}} can be improved to a ``10". We opt for the weaker statement here to keep the proof cleaner.

In the following sections we give a proof of \textcolor{defaultcolor}{Theorem \ref{conflict-theorem}} by colouring $G$ using a semi-random procedure. Similar procedures have been used in many places. The procedures in \cite{asymptotic-list-col, molloy-thron} are particularly close to the one presented here. Our analysis is heavily based on Kim's result \cite{kim95}, as presented in chapter 12 of \cite{graph-colouring-book}. One of the important differences is \textcolor{defaultcolor}{Modification \ref{mainmod}}, where we remove certain ``bad" colours before starting our procedure. This idea has been used in \cite{molloy-thron}, but in our case the situation is slightly more delicate, and so we take greater care with how many colours we choose to remove (see remark \ref{remark1}). Another key difference is a new parameter (see $t_i(v,u,c)$ defined in section \ref{parameters}) which is closely related to the conflict degree. We must ensure that this parameter is sufficiently small to guarantee concentration results of the main parameters considered in chapter 12 of \cite{graph-colouring-book}.

As promised, we now give the details of \textcolor{defaultcolor}{Example \ref{counterexample2}}, which serves as a proof of \textcolor{defaultcolor}{Fact \ref{badfact}}.

\begin{example}\label{counterexample2}
For any $\Delta$, it is easy to find simple graphs with no cycles of length 3 or 4 which are $\Delta$-regular. Let $G$ be such a graph.
By Proposition 6 of \cite{least-con-choose}, such a graph satisfies $\mathrm{ch}_{\nleftrightarrow}(G) > \sqrt{\Delta/\ln\Delta} = \ell$. This means there is at least one edge labelling of $G$ from $[\ell]$ which does not permit a proper conflict colouring of the vertices from $[\ell]$. Let $\tau$ be this edge labelling (to be clear, each label assigned by $\tau$ is a pair of colours). Note that since $G$ is simple, $D(\tau) = 1$.

Now consider the following graph $G'$, constructed by replacing each edge in $G$ by $\ell' = \ell^2$ parallel edges. Consider the set $C = [\ell']$ of $\ell'$ colours. We partition $C$ into $\sqrt{\ell'} = \ell$ contiguous pieces each of size $\ell$. Let $C_i$ be the $i^{th}$ piece of the partition. For each pair $i,j$, we define the following set of pairs of colours: $P_{ij} = \{(c_1,c_2): c_1 \in C_i, c_2 \in C_j\}$. Note that $|P_{ij}| = \ell^2$. Now we define an edge labelling $\tau'$ on $G'$ as follows. For each edge $e \in G$ such that $\tau(e) = (i,j)$, assign a unique pair from $P_{ij}$ to each of the corresponding $\ell^2$ parallel edges in $G'$.

Now we show that $\tau'$ does not permit a proper conflict colouring on $G'$. Suppose to the contrary that $\tau'$ does permit a proper conflict colouring, say $\sigma'$. We define a vertex colouring on $G$ as follows. For each $v \in G$, $\sigma'(v)$ is contained uniquely in some partition $C_i$, so we let $\sigma(v) = i$. So for any edge $e=uv \in G$ with $\tau(e) = (i,j)$, it can not be the case that $\sigma(u) = i$ while $\sigma(v) = j$ simultaneously since it is not possible for $\sigma'(u)$ to be contained in $C_i$ while $\sigma'(v)$ is contained in $C_j$ simultaneously (using the fact that $\sigma'$ is a proper conflict colouring). This means $\tau$ permits a proper conflict colouring of $G$, contrary to assumption.

Note that the maximum degree of $G'$ is $\Delta' = \Delta \times \ell' = \Delta \times \ell^2 = \Delta^2/\ln\Delta$. Using this fact, let us compute $\ell'$ as a function of $\Delta'$. We have:
\begin{align}
    \sqrt{\frac{\Delta'}{\ln\Delta'}} = \sqrt{\frac{\Delta^2}{\ln\Delta\ln\Delta'}} \label{listexpression}
\end{align}
Note that $\ln\Delta' = 2\ln\Delta - \ln\ln\Delta = (2 - o(1))\ln\Delta$. So continuing (\ref{listexpression}) we get:
\begin{align}
    \sqrt{\frac{\Delta^2}{(2-o(1))\ln^2\Delta}} = \frac{1}{\sqrt{2}-o(1)}\frac{\Delta}{\ln\Delta} = \frac{1}{\sqrt{2}-o(1)}\ell'. \nonumber
\end{align}
So after rearranging terms we have $\ell' = (\sqrt{2} - o(1))\sqrt{\Delta'/\ln\Delta'}$.

Now let us examine the relationship between the new conflict degree and the previous. By construction, each edge in $G$ is replaced by $\ell^2$ parallel edges, and each colour conflicting with say $k$ colours in $G$ will be replaced by a colour conflicting with $k \times \ell$ colours in $G'$. I.e. we have $D(\tau') = D(\tau) \times \ell = \sqrt{\ell'}$.

Let us take a moment to realize what was just accomplished. We started with an uncolourable graph and obtained another uncolourable graph, \textcolor{defaultcolor}{increasing} the coefficient on our list sizes by approximately $\sqrt{2}$, but this came at the cost of increasing the conflict degree to the square root of the new list size.

In fact, the same construction can be applied several times. Let $\Delta_i, \ell_i, \tau_i$ be the maximum degree, list size, and edge labelling on the resulting graph after the $i^{th}$ application. Initially we have:
\begin{align}
    \ell_0&=\ell, \nonumber \\
    \Delta_0 &= \Delta, \nonumber \\
    D(\tau_0) &= D(\tau). \nonumber
\end{align}
And by construction we have the following recurrences for $i > 0$:
\begin{align}
    \ell_{i}&=\ell_{i-1}^2, \nonumber \\
    \Delta_{i} &= \Delta_{i-1} \times \ell_{i-1}^2, \nonumber \\
    D(\tau_{i}) &= D(\tau_{i-1}) \times \ell_{i-1}. \nonumber
\end{align}

From these recurrences we have $D(\tau_i) = \ell_i^{1-2^{-i}}$ which the reader can easily verify by reasoning inductively. Also, the same analysis as in (\ref{listexpression}) reveals that $\ell_i = (\sqrt{2} - o(1))^i\sqrt{\Delta_i/\ln\Delta_i}$. So after repeating 4 times, we get an uncolourable graph with maximum degree $\Delta$, lists of size $\ell > (2\sqrt{2}+o(1))\sqrt{\Delta/\ln\Delta}$, and an edge labelling with conflict degree $\ell^{15/16}$. 

More generally, for any constant $\alpha \geq 1$ we have to repeat the construction precisely $\lfloor \log_{\sqrt{2}}\alpha\rfloor + 1$ times to obtain an uncolourable graph with maximum degree $\Delta$ and lists of size at least $\alpha\sqrt{\Delta/\ln\Delta}$. So the function $f$ from \textcolor{defaultcolor}{Fact \ref{badfact}} is given by:

$$ f(\alpha) = 1 - 2^{-\lfloor \log_{\sqrt{2}}\alpha\rfloor - 1}.$$
\end{example}

\section{Probabilistic Tools}

Here we describe three probabilistic tools which will be useful in proving our main theorem. The first is the Lov{\'a}sz Local Lemma which is used several times in the proof. There are several versions of the lemma, the version we present here is not the most powerful but nevertheless it is sufficient for our purposes. The original proof can be found in \cite{locallemmaweak}.

\medskip
\theoremstyle{plain}
\newtheorem*{lll*}{The Lov{\'a}sz Local Lemma}

\begin{lll*}
Consider a set $\mathcal{E}$ of ``bad" events in a probability space, such that for each $A \in \mathcal{E}$:

\begin{enumerate}[label=\alph*)]
    \item $\mathbf{Pr}(A) \leq p < 1$, and 
    \item $A$ is mutually independent of a set of all but at most $d$ of the other events.
\end{enumerate}

\noindent If $4pd \leq 1$ then with positive probability, none of the events in $\mathcal{E}$ occur.
\end{lll*}

\medskip
We use two concentration bounds, the first of which is known as the Chernoff Bound. The original statement can be found in \cite{chernoffstrong}, we use the version which can be found in chapter 5 of \cite{graph-colouring-book}.

\medskip
\theoremstyle{plain}
\newtheorem*{cb}{The Chernoff Bound}

\begin{cb}
    For any $t \geq 0$: $$\mathbf{Pr}(|BIN(n,p) - np| > t) \leq 2e^{-((1+\frac{t}{np})\ln(1+\frac{t}{np})-\frac{t}{np})np}$$
\end{cb}

\medskip
Our second concentration bound is Talagrand's Inequality. The original statement can be found in \cite{talagrandstrong} but once again we use the weaker version which appears in \cite{talagrandweak}.

\medskip
\theoremstyle{plain}
\newtheorem*{ti*}{Talagrand's Inequality}

\begin{ti*}
Let $X$ be a non-negative random variable determined by the independent trials $T_1, T_2, \dots, T_n$. Suppose that for every set of possible outcomes of the trials, we have:

\begin{enumerate}[label=\alph*)]
    \item changing the outcome of any one trial can affect $X$ by at most $q$, and
    \item for each $s > 0$, if $X \geq s$, then there is a set of at most $rs$ trials whose outcomes certify that $X \geq s$ (we make this precise below).
\end{enumerate}

\noindent Then for any $t \geq 0$, we have

$$\mathbf{Pr}(|X - \mathbf{E}(X)| > t + 20q\sqrt{r\mathbf{E}(X)} + 64q^2r) \leq 4e^{-\frac{t^2}{8q^2r(\mathbf{E}(X) + t)}}$$
\end{ti*}

\medskip
Let us be more precise with condition ($b$) above. Suppose there is a set $I \subseteq [n]$, and outcomes $\{t_i : i \in I\}$ such that $X \geq s$ whenever $T_i = t_i$ for each $i \in I$. We say that the set of outcomes $\{t_i : i \in I\}$ for trials $\{T_i : i \in I\}$ \textit{certify} that $X \geq s$.

\section{Proof of Main Theorem}

Fix $\epsilon> 0$. Let $G = (V,E)$ be a multigraph with no cycles of length 3 or 4 and maximum degree $\Delta$. Let $\tau: E \to \mathbb{N} \times \mathbb{N}$ be an edge labelling on $G$ with $D(\tau) \leq \Delta^{\frac{1}{4}\epsilon^2/(\epsilon + 5)^2}$. Instead of specifying a value for $\Delta_0$, throughout the proof we will use the fact that $\Delta$ is a sufficiently large constant.

\bigskip
\noindent
\textbf{A brief comment on terminology and notation.}

For convenience, if $u$ and $v$ are neighbouring vertices we define the set $\Tau(u,v)$ as the set of all labels between $u$ and $v$. That is, $\Tau(u,v)=\{\tau(e): e \textit{ is an edge between u and v} \}$. By convention, all pairs of colours in $\Tau(u,v)$ are ordered so that if $(c,c') \in \Tau(u,v)$, then $c$ corresponds to the colour disallowed for $u$ while $c'$ corresponds to the colour disallowed for $v$. We write $\Tau(v,u)$ to indicate that the pairs are ordered the opposite way. 

Also recall that each label on an edge $e$ constrains the colours assigned to the endpoints of $e$. For this reason it is useful to think of a label on an edge as a constraint, and so throughout the proof we use the terms ``label" and ``constraint" interchangeably. We also speak of the endpoints $u,v$ of a constraint $(c,c')$, and by this we simply mean that $(c,c') \in \Tau(u,v)$. Similarly for neighbouring vertices $u,v$ we write that $u$ and $v$ are constrained by $(c,c')$ to mean that $(c,c') \in \Tau(u,v)$. This terminology will simplify our discussion significantly.

\bigskip
Consider the following semi-random procedure for colouring $G$. As mentioned, during the procedure we maintain a list $L(v)$ for each vertex $v$, which represents the set of colours available to $v$.

\bigskip
\noindent
Wasteful Colouring Procedure ($i^{th}$ iteration)

\begin{enumerate}
    \item Truncate the lists so they all have size $L_{i-1}$ (to be defined later)
    \item For each uncoloured vertex $v$, activate $v$ with probability $\frac{K}{\ln \Delta}$ ($K$ is a constant which will be defined precisely later).
    \item For each activated vertex $v$, assign $v$ a colour uniformly at random from $L(v)$.
    \item For each activated vertex $v$: If $v$ is assigned the colour $c$, for each uncoloured neighbour $u$ of $v$ and each colour $c'$ such that $(c, c') \in \Tau(v, u)$, remove $c'$ from $L(u)$.
    \item For each pair of \textcolor{defaultcolor}{adjacent} vertices $v$ and $u$ which receive colours $c$ and $c'$ respectively, uncolour both $v$ and $u$ if $(c, c') \in \Tau(v,u)$.
    \item Conduct an ``equalizing" coin flip (to be defined later) for each vertex $v$ and colour $c \in L(v)$, removing $c$ from $L(v)$ if it loses the coin flip.
\end{enumerate}

\medskip
Now consider the following parameters which will allow us to analyze the procedure. Each parameter refers to the value at the end of iteration $i$, for $i > 0$. For $i = 0$, each parameter refers to the value before beginning the procedure.

\begin{gather}\label{parameters}
\begin{align}
    L_i(v) &= v'\text{s list of colours} \nonumber \\
    \ell_i(v) &= |L_i(v)| \nonumber \\
    \begin{split}
        N_i(v, u, c) &= \begin{cases} 
                              \{(c, c') \in \Tau(v,u): c' \in L_i(u)\} & \text{if } u \text{ is an uncoloured neighbour of } v \\
                             \emptyset & \text{otherwise}
                           \end{cases}
    \end{split} \nonumber \\
    N_i(v,c) &= \bigsqcup_{u \in N(v)} N_i(v,u,c) \nonumber \\
    t_i(v,u,c) &= |N_i(v,u,c)| \nonumber \\
    t_i(v, c) &= |N_i(v,c)| = \sum_{u \in N(v)}t_i(v,u,c) \nonumber
\end{align}
\end{gather}

\noindent
In the definition of $N_i(v,c)$, the symbol $\bigsqcup$ denotes the disjoint union. Note that $v$ could have the constraint (red, red) between two different neighbours. But these two constraints are different because they have different endpoints, so we take the disjoint union to include both of them.

\begin{observation}\label{mainobserve} \textit{For each $v$, neighbour $u \in N(v)$, and $c \in L(v)$, we initially have $t_0(v,u,c) \leq D(\tau)$. Since this parameter can only decrease, we have $t_i(v,u,c) \leq D(\tau)$ for every relevant $i$. Because of our restriction on the conflict degree this also means $t_i(v,u,c) \leq \Delta^{\frac{1}{4}\epsilon^2/(\epsilon + 5)^2}$.} \\ 

\indent Throughout the proof, many of our results rely on the list sizes being sufficiently large. More specifically, our lemma statements will specify that $L_i \geq \Delta^{\frac{1}{2}\epsilon^2/(\epsilon + 4)^2}$ (as we will see, this requirement is fulfilled by \textcolor{defaultcolor}{Lemma \ref{listsizelargelem}}). 

With list sizes this large, we have:
\begin{align}
    t_i(v,u,c) \leq L_i^{1/2 - \beta}, \text{ where } \beta = \frac{1}{2} - \frac{1}{2}\bigg(\frac{\epsilon + 4}{\epsilon + 5}\bigg)^2. \label{betadef}
\end{align} 
We will also be concerned with the value $T_i$ (to be defined later), which we also require to be sufficiently large. If $T_i \geq \Delta^{\frac{1}{2}\epsilon^2/(\epsilon + 4)^2}$, then again we have:
\begin{align}
    t_i(v,u,c) \leq T_i^{1/2 - \beta}, \label{observebound}
\end{align} 
for the same value of $\beta$. The reader should not dwell on the precise value of $\beta$, the only important detail is that $\beta$ is in the interval $(0, \frac{1}{2})$. \\
\end{observation}

For our procedure to work we require that certain parameters stay reasonably close to their expected values (i.e. we require our parameters to be strongly concentrated). It turns out that the parameter $t_i(v,c)$ is not strongly concentrated as it can be affected by a large amount when $c$ is assigned to $v$ in iteration $i$ (in particular it drops to 0). In fact, assigning colours other than $c$ to $v$ could have the same effect. For this reason, the following closely related parameters will be considered at the end of iteration $i+1$ (the purpose of these parameters is to simply ignore the effect of assigning a colour to $v$):

\begin{gather}
\begin{align}
    \begin{split}
        t'_{i+1}(v, u, c) =& \text{ } 0 \text{, if } u \text{ retained a colour during iteration } i+1, \\
        &\text{ otherwise the number of labels } (c, c') \text{ counted by } t_i(v,u,c) \text{ such that in iteration } i+1 \text{:} \\
        &\qquad \bullet \quad c' \text{ was not removed from } u \text{'s list because of a conflict with a neighbour } w \neq v \\
        &\qquad \bullet \quad c' \text{ was not removed from } u \text{'s list because of an equalizing coin flip}
    \end{split} \nonumber \\
    t'_{i+1}(v,c) = &\sum_{u \in N(v)}t'_{i+1}(v,u,c) \nonumber
\end{align}
\end{gather}

\medskip
Now it turns out that the parameter $t'_{i+1}(v,c)$ is indeed concentrated, as we will show in \textcolor{defaultcolor}{Lemma \ref{concentrationlem}}. Notice that $t_{i+1}(v,c) \leq t'_{i+1}(v,c)$. As we will see, we only need an upper bound on $t_{i+1}(v,c)$. So as long as we have a sufficiently strong bound on $t'_{i+1}(v,c)$ we should be able to achieve this.

\medskip
Now, our goal is to perform several iterations of the procedure so that eventually much of the graph is coloured, and colouring the remaining graph becomes easy. Intuitively, by the end of the procedure we would like that for each vertex $v$, the parameter $\ell_i(v)$ be large while for each colour $c$ in $v$'s list, $t_i(v,c)$ should be quite small. This way, we have several options when choosing a colour for $v$, and each is unlikely to conflict with any neighbour. This will improve our chances of finding a proper conflict colouring. The next lemma makes this precise. 

This result was introduced by Reed \cite{reedlemma}, and has since been seen in modified form in various places (see \textcolor{defaultcolor}{Proposition 2.1} of \cite{asymptotic-list-col}, \textcolor{defaultcolor}{Lemma 3.1} of \cite{molloy-thron}). Likewise, we present a modified version which is suited for our purposes; the proof is nearly identical to the original.


\begin{lem}\label{reedlem}
Let $G = (V,E)$ be a multigraph with edge labelling $\tau: E \to \mathbb{N} \times \mathbb{N}$ and a list of colours $L(v)$ for each vertex $v \in G$. Suppose for each vertex $v \in G$, $|L(v)| \geq \ell$. Also suppose that for each vertex $v$, and each colour $c \in L(v)$, there are at most $\ell/8$ edges $e = vu$ such that $\tau(e)=(c, c')$ where $c' \in L(u)$. Then there is a proper conflict colouring of $G$ from these lists.
\end{lem} 

\begin{proof}
Arbitrarily truncate the lists to be exactly of size $\ell$, and assign each vertex a colour uniformly at random from its list. Let $\sigma$ be the resulting vertex colouring on $G$.

For an edge $e$ with endpoints $v$ and $u$, let $\tau(e) = (c, c')$. We consider the event $A_e$ where $\sigma(v) = c$ and $\sigma(u) = c'$. Let $\mathcal{E}$ be the set of all such events. If $c \notin L(v)$ or $c' \notin L(u)$, then $\mathbf{Pr}(A_{e}) = 0$, and this is mutually independent of any set of events. Otherwise, $\mathbf{Pr}(A_{e}) = 1/\ell^2$.

Consider the following sets of events:
$$D_v = \{A_{f} : f = vv' \neq e, (c_1, c_2) = \tau(f) \textit{ and } c_2 \in L(v')\}$$
$$D_u = \{A_{f} : f = uu' \neq e, (c_1, c_2) = \tau(f) \textit{ and } c_2 \in L(u')\}$$

\medskip
\noindent
(note that these sets could contain events indexed by edges parallel to $e$). Observe that $A_e$ is mutually independent of all events in $\mathcal{E} - D_v - D_u$ by the Mutual Independence Principle (see chapter 4 of \cite{graph-colouring-book}).

By assumption, for each colour $c_1 \in L(v)$, there could be at most $\ell/8$ edges $e=vv'$ such that $\tau(e) = (c_1, c_2)$ and $c_2 \in L(v')$. Hence, summing over all colours in $v$'s list gives a bound of $\ell^2/8$ on the size of $D_v$, and similarly for $D_u$. So each bad event occurs with probability at most $p=1/\ell^2$ and is \textcolor{defaultcolor}{mutually independent of all but} at most $d=\ell^2/4$ other events. Therefore, $4pd \leq 4(1/\ell^2)(\ell^2/4) \leq 1$. So we can apply the Lov{\'a}sz Local Lemma \textcolor{defaultcolor}{to obtain a colouring $\sigma$ such that there is no edge $e$ with endpoints $v$ and $u$ and $\tau(e) = (c,c')$ where $\sigma(v) = c$ and $\sigma(u) = c'$, i.e. $\sigma$ is a proper conflict colouring}.
\end{proof}

The previous lemma permits us to focus on the parameters $\ell_i(v)$ and $t_i(v,c)$, but instead of keeping track of the values $\ell_i(v)$ and $t_i(v,c)$ for each $v, c$, we focus on their extreme values. We recursively define appropriate $L_i$ and $T_i$ which we would like to serve as a lower bound for $\ell_i(v)$ and an upper bound for $t_i(v,c)$, respectively. Moreover, the value $t_i(v,u,c)$ will also be of interest to us, as if this value is too large then $\ell_i(v)$ and $t_i(v,c)$ will not be sufficiently concentrated.

More precisely, we will show that with positive probability the following property holds at the end of iteration $i$:

\bigskip
\noindent
\textbf{Property P(i):} For each uncoloured vertex $v$, uncoloured neighbour $u$ of $v$, and each colour $c \in L_i(v)$,

$$\ell_i(v) \geq L_i$$

$$t_i(v,c) \leq T_i$$

\bigskip
So if for some iteration, $i$, property $P(i)$ holds and $T_i \leq \frac{1}{8}L_i$ then the colouring can be completed by \textcolor{defaultcolor}{Lemma \ref{reedlem}}.

\medskip
As mentioned in step 6 of the procedure, if $c$ is in $v$'s list near the end of an iteration (just before step 6), we remove $c$ with some probability. We do this in such a way that forces the probability of a vertex retaining a given colour during the entirety of the iteration to be the same across vertices and colours. Now we define $\mathrm{Keep}_i(v,c)$ \textcolor{defaultcolor}{to} be the probability that $c$ stays in $v$'s list by the beginning of step 6 of iteration $i+1$\textcolor{defaultcolor}{, conditional on the event that $c$ is in $v$'s list after step 1 of iteration $i+1$}. For a neighbour $u$ of $v$, there are $t_i(v,u,c)$ colours in $u$'s list which would cause $c$ to be removed from $v$'s list if any one of them is assigned to $u$. Since the size of $u$'s list is precisely $L_i$ by step 3 in the procedure, and $u$ is activated with probability $K/\ln\Delta$, we arrive at the following expression:

\begin{gather}
    \label{eqn:keepvc}
    \mathrm{Keep}_i(v,c) = \prod_{u \in N(v)}\Big(1-\frac{K}{\ln\Delta} \times \frac{t_i(v, u, c)}{L_i}\Big)
\end{gather}

\textcolor{defaultcolor}{
Now we define the following value which we claim is a lower bound on $\mathrm{Keep}_i(v,c)$ and will be easier to work with.}

\medskip
\noindent
\begin{align}
    \label{eqn:keep}
    \mathrm{Keep}_i := \Big(1-\frac{K(1+\epsilon e^{-\epsilon}/30)}{L_i\ln\Delta}\Big)^{T_i}
\end{align}

\begin{claim}\label{keepclaim}
If $P(i)$ holds, $T_i \geq \frac{1}{8}L_i$, and if $T_i,L_i \geq \Delta^{\frac{1}{2}\epsilon^2/(\epsilon + 4)^2}$, then $\mathrm{Keep}_i(v,c) \geq \mathrm{Keep}_i$ for $\Delta$ sufficiently large.
\end{claim}

\begin{proof}
\textcolor{defaultcolor}{
It suffices to show the following for $\Delta$ sufficiently large:}
\begin{align}
    \label{eqn:intermediatekeepbound}
    \Big(1 - \frac{K t_i(v,u,c)}{L_i\ln\Delta}\Big) \geq \Big(1 - \frac{K(1+\epsilon e^{-\epsilon}/30)}{L_i\ln\Delta}\Big)^{t_i(v,u,c)}
\end{align}
\noindent
\textcolor{defaultcolor}{
As then we have:
\begin{align}
        \mathrm{Keep}_i(v,c) & = \prod_{u \in N(v)}\Big(1-\frac{K}{\ln\Delta} \times \frac{t_i(v, u, c)}{L_i}\Big) \nonumber \tag*{by (\ref{eqn:keepvc})} \\
        & \geq \prod_{u \in N(v)}\Big(1-\frac{K(1+\epsilon e^{-\epsilon}/30)}{L_i\ln\Delta}\Big)^{t_i(v,u,c)} \nonumber \tag*{by (\ref{eqn:intermediatekeepbound})} \\
        & = \Big(1-\frac{K(1+\epsilon e^{-\epsilon}/30)}{L_i\ln\Delta}\Big)^{\sum_{u \in N(v)}t_i(v,u,c)} \nonumber \\
        & = \Big(1-\frac{K(1+\epsilon e^{-\epsilon}/30)}{L_i\ln\Delta}\Big)^{t_i(v,c)} \nonumber \\
        & \geq \Big(1-\frac{K(1+\epsilon e^{-\epsilon}/30)}{L_i\ln\Delta}\Big)^{T_i} \nonumber \tag*{by property $P(i)$} \\
        & = \mathrm{Keep}_i \tag*{by (\ref{eqn:keep})} \nonumber
\end{align}}

\noindent
\textcolor{defaultcolor}{
We can show (\ref{eqn:intermediatekeepbound}) by considering two well-known facts:
\begin{align}
    \label{eqn:uppertaylorbound}
    e^{-x} & \leq 1 - x + \frac{x^2}{2} && \text{for all non-negative $x$} \\
    \label{eqn:lowertaylorbound}
    e^{-x} & \geq 1 - x && \text{for all $x$}
\end{align}}

\noindent
\textcolor{defaultcolor}{
This comes from the fact that for Taylor polynomials of $e^{-x}$ with even degree, the remainder is non-positive for all positive $x$. Similarly, the Taylor polynomials of $e^{-x}$ with odd degree have a remainder that is non-negative for all $x$.}
\textcolor{defaultcolor}{
Now consider $$x = \frac{K(1+\epsilon e^{-\epsilon}/30)t_i(v,u,c)}{L_i\ln\Delta}$$
and note that by (\ref{betadef}) and since $L_i \geq \Delta^{\frac{1}{2}\epsilon^2/(\epsilon+4)^2}$ by assumption, $x$ tends to $0$ as $\Delta$ tends to infinity. So we have for $\Delta$ sufficiently large:
\begin{align}
    \exp\bigg(-\frac{K(1+\epsilon e^{-\epsilon}/30)t_i(v,u,c)}{L_i\ln\Delta}\bigg) & \leq 1 - x + \frac{x^2}{2} \tag*{by (\ref{eqn:uppertaylorbound})} \\
    & \leq 1-\frac{x}{1+\epsilon e^{-\epsilon}/30} \nonumber \tag*{since $x$ dominates over $\frac{x^2}{2}$ for large $\Delta$} \\
    & = 1 -\frac{Kt_i(v,u,c)}{L_i\ln\Delta} \label{eqn:intermediatekeepbound2}
\end{align}
\noindent
On the other hand, 
\begin{align}
    \exp\bigg(-\frac{K(1+\epsilon e^{-\epsilon}/30)t_i(v,u,c)}{L_i\ln\Delta}\bigg) & \geq \exp\bigg(-\frac{K(1+\epsilon e^{-\epsilon}/30)}{L_i\ln\Delta}\bigg)^{t_i(v,u,c)} \nonumber \\
    & \geq \bigg(1 - \frac{K(1+\epsilon e^{-\epsilon}/30)}{L_i\ln\Delta}\bigg)^{t_i(v,u,c)} \tag*{by (\ref{eqn:lowertaylorbound})}
\end{align}}
\noindent
which proves the claim.
\end{proof}

\begin{remark}
We remark that \textcolor{defaultcolor}{Claim \ref{keepclaim}} remains true even if we replace the expression $\epsilon e^{-\epsilon}$ by any positive constant. Although this expression may seem bizarre, we choose it because we require an expression that is quite small for all values of $\epsilon$, even values which are very large.
\end{remark}

\medskip
\noindent
\textbf{Equalizing coin flips.}

We use \textcolor{defaultcolor}{Claim \ref{keepclaim}} to define our equalizing coin flips. Define the following probability: $$\mathrm{Eq}_i(v,c) := 1 - \mathrm{Keep}_i/\mathrm{Keep}_i(v,c).$$ 

We perform the equalizing coin flip described in step 6 of the procedure by removing any colour $c$ from $v$'s list with probability $\mathrm{Eq}_i(v,c)$. This way, we ensure that the probability of $c$ remaining in $v$'s list by the end of iteration $i+1$ is precisely $(1-\mathrm{Eq}_i(v,c)) \times \mathrm{Keep}_i(v,c) = \mathrm{Keep}_i$.

\medskip
Now initially we have that for each $v$, $\ell_0(v) = (2\sqrt{2}+ \epsilon)\sqrt{\Delta/\ln\Delta}$. However, for a colour $c$ in $v$'s list, $t_0(v, c)$ could be as high as $\Delta$. Recall that we would like $t_0(v,c)$ to be quite small, and it turns out that this is not good enough for our purposes, so we remove all colours $c$ from $v$'s list for which $t_0(v,c)$ exceeds some threshold (a similar technique has been used in \cite{molloy-thron}). The following claim makes this precise.

\begin{claim}\label{truncationlem}
For each vertex $v$, the number of colours $c$ for which $t_0(v,c) > \frac{1}{\sqrt{2}+\epsilon/2}\sqrt{\Delta\ln\Delta}$ is at most $(\sqrt{2} + \epsilon/2)\sqrt{\Delta/\ln\Delta}$.
\end{claim}

\begin{proof}
Suppose to the contrary that strictly more than $(\sqrt{2} + \epsilon/2)\sqrt{\Delta/\ln\Delta}$ colours $c$ achieve a value of $t_0(v,c) > \frac{1}{\sqrt{2} + \epsilon/2}\sqrt{\Delta\ln\Delta}$. Then the number of edges incident to $v$ is is strictly more than $(\sqrt{2} + \epsilon/2)\sqrt{\Delta/\ln\Delta} \times \frac{1}{\sqrt{2} + \epsilon/2}\sqrt{\Delta\ln\Delta} = \Delta$, a contradiction.
\end{proof}

The previous claim allows us to modify our graph by removing ``bad" colours before starting the procedure. In the worst case we are left with at least half of the colours we originally started with.

\medskip
\begin{modification}\label{mainmod}
For each vertex $v$ and colour $c$ in $v$'s list, remove $c$ from $v$'s list if $t_0(v,c) > \frac{1}{\sqrt{2} + \epsilon/2}\sqrt{\Delta\ln\Delta}$.
\end{modification}

\begin{remark}\label{remark1}
Eliminating half of $v$'s colours effectively means that we must start with twice the number of colours we really need. In other words, if such ``bad" colours posed no problem, and hence there was no need to remove them, we could get away with a leading constant of $\sqrt{2} + \epsilon$ in \textcolor{defaultcolor}{Theorem \ref{adaptable-theorem} and Theorem \ref{conflict-theorem}}. The reader may wonder if imposing a larger threshold in \textcolor{defaultcolor}{Modification \ref{mainmod}} (thus removing less colours) gives a smaller leading constant. Unfortunately, there is an upper limit to what threshold can be used as we must ensure $t_0(v,c)$ is not too large. Conversely, if the threshold is too relaxed we end up removing too many colours. It turns out that the threshold proposed is optimal in the sense that it strikes the right balance between these two extremes.
\end{remark}

Finally, we are ready to give recursive definitions for $L_i$ and $T_i$. Keep in mind that in the following definitions, $\beta$ is the constant specified by (\ref{betadef}). Initially, by \textcolor{defaultcolor}{Modification \ref{mainmod}} and \textcolor{defaultcolor}{Claim \ref{truncationlem}} we have the following: $$L_0 = (\sqrt{2} + \epsilon/2)\sqrt{\Delta/\ln\Delta}, \ \ \ \ \ T_0 = \frac{1}{\sqrt{2} + \epsilon/2}\sqrt{\Delta\ln\Delta},$$ and we recursively define:

\begin{gather}\label{eqn:eqs1}
    \begin{align}
        L_{i+1} &= L_i \times \mathrm{Keep}_i - L_i^{1-\beta/2} \nonumber
        \\
        \\
        T_{i+1} &= T_i\Big(1-\frac{K}{\ln\Delta}\mathrm{Keep}_i\Big) \times \mathrm{Keep}_i + T_i^{1-\beta/2} \nonumber
    \end{align}
\end{gather}

\bigskip
We also give recursive definitions for $L'_i, T'_i$, which are closely related to $L_i$ and $T_i$. Initially we have the following: $$L'_0 = L_0, \ \ \ \ \ T'_0 = T_0,$$
and we recursively define:

\begin{gather}\label{eqn:eqs2}
    \begin{align}
        L'_{i+1} &= L'_i \times \mathrm{Keep}_i \nonumber
        \\
        \\
        T'_{i+1} &= T'_i\Big(1-\frac{K}{\ln\Delta}\mathrm{Keep}_i\Big) \times \mathrm{Keep}_i \nonumber
    \end{align}
\end{gather}

\bigskip
$L'_i$ and $T'_i$ are much simpler to analyze, and we can show that they do not stray too far from values $L_i$ and $T_i$, respectively (see \textcolor{defaultcolor}{Lemma \ref{closeequationslem}}). This allows us to focus our attention on analyzing these simpler equations. Now we move on to proving several lemmas, which will allow us to complete the proof of the main theorem.

Throughout the proof, we use the fact that the ratio $T_i/L_i$ is decreasing. We prove this in the following lemma.


\begin{lem}\label{decreasinglem}
If for all $j < i, L_j, T_j \geq \Delta^{\frac{1}{2}\epsilon^2/(\epsilon + 4)^2}$ and $T_j \geq \frac{1}{8}L_j$ then $T_i/L_i < T_{i-1}/L_{i-1}$
\end{lem}

\begin{proof}
The proof is by induction. We assume inductively that: $$T_i/L_i < T_{i-1}/L_{i-1} < \cdots < T_0/L_0 < \ln\Delta$$ and we shall abbreviate ``inductive hypothesis" by $IH$. Note that by $IH$, we have the following useful fact:
\begin{align}
    \mathrm{Keep}_i = \Big(1-\frac{K(1+\epsilon e^{-\epsilon}/30)}{L_i\ln\Delta}\Big)^{T_i} \geq e^{-\frac{K(1+\epsilon e^{-\epsilon}/20)T_i}{L_i\ln\Delta}} \geq e^{-\frac{K(1+\epsilon e^{-\epsilon}/20)T_0}{L_0\ln\Delta}} = \Omega(1). \label{keeplowerbound}
\end{align}

Now we consider $i+1$. By definition (\ref{eqn:eqs1}) we have: $$L_{i+1} = L_i \times \Big(\mathrm{Keep}_i -L_i^{-\beta/2}\Big).$$

\noindent
Also, we have:
\begin{align}
        T_{i+1} &= T_i \times \mathrm{Keep}_i - T_i \times \frac{K}{\ln\Delta}\mathrm{Keep}^2_i + T_i^{1-\beta/2} \nonumber \tag*{by definition (\ref{eqn:eqs1})} \\
        &\leq T_i \times \mathrm{Keep}_i - \mathrm{Keep}_i^2 \times T_i^{1-\beta/8} + T_i^{1-\beta/2} \nonumber \tag*{since $T_i \geq \Delta^{\frac{1}{2}\epsilon^2/(\epsilon + 4)^2}$} \\
        &\leq T_i \times \mathrm{Keep}_i - T_i^{1-\beta/4}
        \nonumber \tag*{since $\mathrm{Keep}_i = \Omega(1)$ by (\ref{keeplowerbound})}\\
        &= T_i \times \mathrm{Keep}_i - T_i^{1-\beta/2} \times T_i^{\beta/4} \nonumber \\
        &\leq T_i \times \mathrm{Keep}_i - T_i^{1-\beta/2} \times (T_i/L_i)^{\beta/2} \nonumber \tag*{since $T_i/L_i < \ln\Delta$ by $IH$} \\
        &= T_i \times \Big(\mathrm{Keep}_i - L_i^{-\beta/2}\Big) \nonumber
\end{align}

\noindent
so $T_{i+1}/L_{i+1} < T_i/L_i$.
\end{proof}

Notice that since $T_0/L_0 < \ln\Delta$ initially, the previous lemma implies $T_i/L_i < \ln\Delta$ for every relevant $i$. It will also be useful to have upper and lower bounds on $\mathrm{Keep}_i$, which we will see in the next lemma.


\begin{lem}\label{keepboundlem}
If for all $j \leq i, L_j,T_j \geq \Delta^{\frac{1}{2}\epsilon^2/(\epsilon + 4)^2}$ and $T_j \geq \frac{1}{8}L_j$ then $e^{-\frac{K(1+\epsilon e^{-\epsilon}/20)}{(\sqrt{2} + \epsilon/2)^2}} \leq \mathrm{Keep}_i \leq 1 - \textcolor{defaultcolor}{\frac{K}{10\ln\Delta}}$
\end{lem}

\begin{proof}
For the lower bound, we make use of \textcolor{defaultcolor}{Lemma \ref{decreasinglem}} as well as the same argument of (\ref{keeplowerbound}) to get:
$$\mathrm{Keep}_i = \Big(1-\frac{K(1+\epsilon e^{-\epsilon}/30)}{L_i\ln\Delta}\Big)^{T_i} \geq e^{-\frac{K(1+\epsilon e^{-\epsilon}/20)T_i}{L_i\ln\Delta}} \geq e^{-\frac{K(1+\epsilon e^{-\epsilon}/20)T_0}{L_0\ln\Delta}} = e^{-\frac{K(1+\epsilon e^{-\epsilon}/20)}{(\sqrt{2} + \epsilon/2)^2}}.$$ \\
\indent For the upper bound, we have: $$\mathrm{Keep}_i = \Big(1-\frac{K(1+\epsilon e^{-\epsilon}/30)}{L_i\ln\Delta}\Big)^{T_i} \leq e^{-\frac{K(1+\epsilon e^{-\epsilon}/30)T_i}{L_i\ln\Delta}} \leq e^{-\frac{K(1+\epsilon e^{-\epsilon}/30)}{8\ln\Delta}} \leq 1 - \textcolor{defaultcolor}{\frac{K}{10\ln\Delta}}$$ using the assumption that $T_i \geq \frac{1}{8}L_i$
\end{proof}

\begin{remark}\label{Kremark}
We remark that by choice of $K$, the lower bound given by \textcolor{defaultcolor}{Lemma \ref{keepboundlem}} is at least $\frac{1}{2}$. We will see this in \textcolor{defaultcolor}{Lemma \ref{listsizelargelem}} when we precisely define $K$ (see definition (\ref{Kdef})).
\end{remark}

The next step is to compute the expected values of our parameters, which we will see in the following lemma.


\begin{lem}\label{expectationlem}
If $P(i)$ holds, $T_i \geq \frac{1}{8}L_i$, and if $T_i,L_i \geq \Delta^{\frac{1}{2}\epsilon^2/(\epsilon + 4)^2}$, then for every uncoloured vertex $v$ and colour $c \in L_i(v)$
\begin{enumerate}[label=\alph*)]
    \item $\mathbf{E}(\ell_{i+1}(v)) = L_i \times \mathrm{Keep}_i$;
    \item $\mathbf{E}(t'_{i+1}(v,c)) \leq t_i(v,c)(1- \frac{K}{\ln\Delta} \mathrm{Keep}_i) \times \mathrm{Keep}_i + T_i^{3/4}$.
\end{enumerate}
\end{lem}

\begin{proof}
\
\begin{enumerate}[wide=\parindent]
    \item [$a)$]
    By property $P(i)$, $\ell_i(v) \geq L_i$. So the truncation step in the procedure is justified, and the size of $v$'s list will be precisely $L_i$ just after the first step of iteration $i+1$. Because of our equalizing coin flip step in the procedure, the probability that a colour $c$ in $v$'s list remains in $v$'s list by the end of the iteration is precisely $\mathrm{Keep}_i$, and so the rest follows by linearity of expectation.
    \item [$b)$]
    Consider any constraint $(c,c') \in N_i(v,c)$ with endpoints $v,u$ where $u$ is not coloured. By linearity of expectation, it suffices to give a bound on the probability of $u$ not retaining a colour and $c'$ not being lost from $u$'s list because of either:
    \begin{itemize}[leftmargin=.6in]
        \item an equalizing coin flip, or,
        \item a conflict with a colour assigned to a neighbour $w \neq v$
    \end{itemize}
    
    \noindent
    during iteration $i+1$. Call this event $A$.
    
    We consider 3 cases:
    \begin{itemize}[leftmargin=.6in]
        \item when $u$ is not activated
        \item when $u$ is activated and not assigned $c'$
        \item when $u$ is activated and assigned $c'$
    \end{itemize}
    and we compute the probability of $A$ occurring, conditional on each case.
    
    \medskip
    \noindent
    \textit{Case 1: $u$ is not activated}
    
    \indent Call this case $B_1$. We are interested in $\mathbf{Pr}(A|B_1)$. If $u$ is not activated (and therefore not assigned a colour), this is precisely the probability that $c'$ remains in $u$'s list or is lost due to a conflict with $v$. By the union bound we get:
    \begin{align}
        \mathbf{Pr}(A|B_1) = \mathrm{Keep}_i + \frac{K}{\ln\Delta} \frac{t_i(u,v,c')}{L_i} \leq \mathrm{Keep}_i + L_i^{-1/2} \label{case1prob}
    \end{align} by (\ref{betadef}).

    \medskip
    \noindent
    \textit{Case 2: $u$ is activated and not assigned $c'$}
    
    \indent Call this case $B_2$.  We are interested in $\mathbf{Pr}(A|B_2)$. We introduce two new events: let $A_1$ be the event that $c'$ remains in $L_{i+1}(u)$ or is lost due to a conflict with $v$, and let $A_2$ be the event that for at least one neighbour $w$ of $u$, $w$ is activated and the colour assigned to $w$ makes $u$ \textcolor{defaultcolor}{uncoloured}. Note that $A = A_1 \cap A_2$, so our strategy for computing $\mathbf{Pr}(A|B_2)$ is to compute $\mathbf{Pr}(A_1|B_2) \times \mathbf{Pr}(A_2|A_1 \cap B_2)$ instead. The reader can easily verify that these two expressions are equal by using the definition of conditional probabilities. \\
    \indent First we consider $\mathbf{Pr}(A_1|B_2)$. Note that conditioning on $u$ being assigned a colour other than $c'$ has no effect on the colours assigned on vertices in $N(u)$. So by the same argument as (\ref{case1prob}) we have:
    \begin{align}
        \mathbf{Pr}(A_1|B_2) \leq\mathrm{Keep}_i + L_i^{-1/2}. \label{case2prob1}
    \end{align}
    
    Now we consider $\mathbf{Pr}(A_2|A_1 \cap B_2)$. For $\gamma \in L_i(u) - c'$, define $B_{2,\gamma}$ to be the event that $u$ is activated and assigned $\gamma$. I.e. $B_2 = \bigcup_{\gamma} B_{2,\gamma}$ is a union of disjoint events, each equally likely to occur with some probability, say $p$. Also note that each $B_{2,\gamma}$ is independent from $A_1$. From this we get the following: 
    \begin{align}
        \mathbf{Pr}(A_2|A_1 \cap B_2) = \frac{\sum_{\gamma}\mathbf{Pr}(A_2 \cap A_1 \cap B_{2,\gamma})}{(L_i-1)p \times \mathbf{Pr}(A_1)} \leq \frac{\max_{\gamma}\mathbf{Pr}(A_2 \cap A_1 \cap B_{2,\gamma})}{p \times \mathbf{Pr}(A_1)} = \max_{\gamma}\mathbf{Pr}(A_2|A_1 \cap B_{2,\gamma}). \label{strategicbound}
    \end{align}
    I.e. it suffices to achieve a bound on $\mathbf{Pr}(A_2|A_1 \cap B_{2,\gamma})$. Fix $\gamma \in L_i(u) - c'$. To achieve our bound, for a neighbour $w$ of $u$ we define $A_{2,w}$ to be the event that $w$ is activated and assigned a colour which makes $u$ lose $\gamma$. Then we have the following:
    \begin{align}
        \mathbf{Pr}(A_2|A_1 \cap B_{2,\gamma}) = \bigg(1 - \prod_{w \in N(u)}(1 - \mathbf{Pr}(A_{2,w}|A_1 \cap B_{2,\gamma})\bigg). \label{probformula}
    \end{align} \\
    \indent We consider two subcases. When $w = v$, we have:
    \begin{align}
        \mathbf{Pr}(A_{2,v}|A_1 \cap B_{2,\gamma}) = \frac{K}{\ln\Delta} \times \frac{t_i(u,v,\gamma)}{L_i}. \label{case2subprob1}
    \end{align}
    For any other neighbour $w \in N(u) - v$, we should be more careful since a colour assignment on $w$ could make $u$ lose $c'$ thus violating event $A_1$. Since the activations and colour assignments are independent over different vertices, and since $\gamma \neq c'$, we simply have:
    $$ \mathbf{Pr}(A_{2,w}|A_1 \cap B_{2,\gamma}) = \mathbf{Pr}((w \textit{ is activated and assigned a bad colour for } u, \gamma) | D_w),$$
    where $D_w$ is the event that $w$ is not assigned a colour which makes $u$ lose $c'$. So we get:
    \begin{align}
        &\mathbf{Pr}((w \textit{ is activated and assigned a bad colour for } u, \gamma) \cap D_w)/\mathbf{Pr}(D_w) \nonumber \\
        \leq \ & \mathbf{Pr}(w \textit{ is activated and assigned a bad colour for } u, \gamma)/\mathbf{Pr}(D_w) \nonumber \\
        = \ & \Bigg(\frac{K}{\ln\Delta} \times \frac{t_i(u,w,\gamma)}{L_i} \Bigg) \times \frac{1}{\Big(1 - \frac{K}{\ln\Delta} \times \frac{t_i(u,w,c')}{L_i} \Big)}. \label{rwbound}
    \end{align}
  
    \noindent
    Note that by (\ref{betadef}), the second factor tends to $1$ from above. This means that for $\Delta$ sufficiently large, it is at most $(1 + \epsilon e^{-\epsilon}/40)$. So continuing (\ref{rwbound}) we get:
        \begin{align}
            \mathbf{Pr}(A_{2,w}|A_1 \cap B_{2,\gamma}) &\leq \Bigg(\frac{K}{\ln\Delta} \times \frac{t_i(u,w,\gamma)}{L_i} \Bigg) \times \frac{1}{\Big(1 - \frac{K}{\ln\Delta} \times \frac{t_i(u,w,c')}{L_i} \Big)} \nonumber \\
            &\leq (1 + \epsilon e^{-\epsilon}/40)\Bigg(\frac{K}{\ln\Delta} \times \frac{t_i(u,w,\gamma)}{L_i} \Bigg). \label{case2subprob2}
        \end{align}
    \noindent
    So by (\ref{case2subprob1}), (\ref{case2subprob2}) and (\ref{probformula}) we get:
        \begin{align}
             \mathbf{Pr}(A_2|A_1 \cap B_{2,\gamma}) & \leq 1 - \bigg(1-\frac{Kt_i(u,v,\gamma)}{L_i\ln\Delta}\bigg) \times \prod_{w \in N(u) - v}\bigg(1 - (1 + \epsilon e^{-\epsilon}/40)\frac{Kt_i(u,w,\gamma)}{L_i\ln\Delta}\bigg) \nonumber \\
            & \leq 1 - \prod_{w \in N(u)}\bigg(1 - \frac{K(1 + \epsilon e^{-\epsilon}/40)t_i(u,w,\gamma)}{L_i\ln\Delta}\bigg) \nonumber.
        \end{align} 
    \noindent
    Now using the same argument as in \textcolor{defaultcolor}{Claim \ref{keepclaim}} we can see that this is at most $ 1 - \mathrm{Keep}_i$. So by (\ref{strategicbound}) we have:
    \begin{align}
        \mathbf{Pr}(A_2|A_1 \cap B_2) \leq 1 - \mathrm{Keep}_i. \label{case2prob2}
    \end{align}
    Finally, by (\ref{case2prob1}) and (\ref{case2prob2}), we have: 
    \begin{align}
        \mathbf{Pr}(A|B_2) = (\mathrm{Keep}_i + L_i^{-1/2})(1 - \mathrm{Keep}_i). \label{case2prob}
    \end{align}
    
    \medskip
    \noindent
    \textit{Case 3: $u$ is activated and assigned $c'$}
    
    \indent Call this case $B_3$. We are interested in $\mathbf{Pr}(A|B_3)$. The only scenario where $u$ remains uncoloured and $c'$ is either lost due to a conflict with $v$ or remains in $L_{i+1}(u)$, is if $c'$ is lost due to a conflict with $v$. This happens with probability 
    \begin{align}
        \mathbf{Pr}(A|B_3) = \frac{K}{\ln\Delta}\frac{t_i(u,v,c')}{L_i} < L_i^{-1/2} \label{case3prob}
    \end{align}
    by (\ref{betadef}). \\
    \indent Therefore, by the law of total probability and (\ref{case1prob}), (\ref{case2prob}), (\ref{case3prob}), we get that the probability of the event $A$ occurring is:
        \begin{align}
           &\ \mathbf{Pr}\big(B_1\big)\times\mathbf{Pr}\big(A|B_1\big) + \mathbf{Pr}\big(B_2\big)\times\mathbf{Pr}\big(A|B_2\big) + \mathbf{Pr}\big(B_3\big)\times\mathbf{Pr}\big(A|B_3\big) \nonumber \\
            \leq &\ \Bigg(1 - \frac{K}{\ln\Delta}\Bigg) \times \Bigg(\mathrm{Keep}_i + L_i^{-1/2}\Bigg) + \frac{K}{\ln\Delta} \times \frac{L_i-1}{L_i} \times \Bigg(\mathrm{Keep}_i + L_i^{-1/2}\Bigg) \Bigg(1 - \mathrm{Keep}_i\Bigg) + \frac{K}{\ln\Delta}\times\frac{L_i^{-1/2}}{L_i} \nonumber \\
            \leq &\ \Bigg(1 - \frac{K}{\ln\Delta}\Bigg) \times \mathrm{Keep}_i + \frac{K}{\ln\Delta} \times \mathrm{Keep}_i \times \Bigg(1 - \mathrm{Keep}_i\Bigg) + 3L_i^{-1/2} \nonumber \\
            = &\ \mathrm{Keep}_i - \frac{K}{\ln\Delta}\mathrm{Keep}_i + \frac{K}{\ln\Delta}\mathrm{Keep}_i - \frac{K}{\ln\Delta}\mathrm{Keep}^2_i + 3L_i^{-1/2} \nonumber \\
            \leq &\ \bigg(1 - \frac{K}{\ln\Delta} \times \mathrm{Keep}_i \Bigg) \times \mathrm{Keep}_i + L_i^{-1/3}. \tag*{since by assumption $L_i \geq \Delta^{\frac{1}{2}\epsilon^2/(\epsilon + 4)^2}$} \nonumber
        \end{align}
        
        \noindent for $\Delta$ sufficiently large.
        
        Now by linearity of expectation we have \begin{align}
            \mathbf{E}(t'_{i+1}(v,c)) &\leq t_i(v,c)\bigg(1-\frac{K}{\ln\Delta}\mathrm{Keep}_i\bigg) \times \mathrm{Keep}_i + t_i(v,c)L_i^{-1/3} \nonumber \\
            &\leq t_i(v,c)\bigg(1-\frac{K}{\ln\Delta}\mathrm{Keep}_i\bigg) \times \mathrm{Keep}_i + T_iL_i^{-1/3} \nonumber \tag*{by property $P(i)$} \\
            &= t_i(v,c)\bigg(1-\frac{K}{\ln\Delta}\mathrm{Keep}_i\bigg) \times \mathrm{Keep}_i + T_i^{2/3} \bigg(\frac{T_i}{L_i}\bigg)^{1/3} \nonumber \\
            &\leq t_i(v,c)\bigg(1-\frac{K}{\ln\Delta}\mathrm{Keep}_i\bigg) \times \mathrm{Keep}_i + T_i^{2/3}\big(\ln\Delta\big)^{1/3} \tag*{since $T_i/L_i < \ln\Delta$ by \textcolor{defaultcolor}{Lemma \ref{decreasinglem}}} \nonumber \\
            &\leq t_i(v,c)\bigg(1-\frac{K}{\ln\Delta}\mathrm{Keep}_i\bigg) \times \mathrm{Keep}_i + T_i^{3/4} \tag*{since by assumption $T_i \geq \Delta^{\frac{1}{2}\epsilon^2/(\epsilon + 4)^2}$} \nonumber
        \end{align}
        
        \noindent for $\Delta$ sufficiently large.
\end{enumerate}

\end{proof}

The next step is to show concentration results for our parameters.

\begin{lem}\label{concentrationlem}
If $P(i)$ holds, $T_i \geq \frac{1}{8}L_i$,and if $T_i, L_i \geq \Delta^{\frac{1}{2}\epsilon^2/(\epsilon + 4)^2}$, then for any uncoloured vertex $v$ and colour $c \in L_i(v)$,
\begin{enumerate}[label=\alph*)]
    \item $\mathbf{Pr}\Big(|\ell_{i+1}(v)-\mathbf{E}(\ell_{i+1}(v))| > L_i^{1-\beta/2}\Big) < \Delta^{-\ln\Delta}$;
    \item $\mathbf{Pr}\Big(|t'_{i+1}(v,c)-\mathbf{E}(t'_{i+1}(v,c))| > \frac{1}{2}T_i^{1-\beta/2}\Big) < \Delta^{-\ln\Delta}$.
\end{enumerate}
\end{lem}

\begin{proof} To show that our parameters are strongly concentrated, we will make use of Talagrand's Inequality. Recall that Talagrand's Inequality requires that a random variable be determined by a set of independent trials. 

The random choices made in an iteration are as follows. For each vertex we choose whether or not to activate it. We assign a random colour to each activated vertex. Finally, we conduct an equalizing coin flip for each $v$ and $c$ in $v$'s list, removing $c$ from $v$'s list if it loses the coin flip. A problem arises when applying Talagrand's Inequality, since these choices are not made independently of one another. For instance, the choice of whether to assign a random colour to a vertex depends on the activation choice for that vertex. Also, the choice of whether to conduct an equalizing coin flip for a given colour and vertex $v$ depends on the fact that this colour was not removed from $v$'s list earlier in the iteration, and this is determined by the activation and colour assignments on $v$'s neighbours. We add a set of dummy choices to fix this. We choose a colour for each unactivated vertex and conduct an equalizing coin flip for each $v$ and $c$ not in $v$'s list. Of course we do not assign the chosen colour to the unactivated vertex and there is no need to remove $c$ from $v$'s list if it loses the coin flip, so these dummy choices have no effect on the procedure. Their purpose is to make the set of random choices independent from one another, and this will allow us to apply Talagrand's Inequality.
\
\begin{enumerate}[wide=\parindent]
    \item[$a)$]
    We focus on showing that the number of colours removed from $v's$ list during iteration $i+1$, $\bar \ell$, is highly concentrated. \\
    \indent Changing the activation or colour assignment to any vertex $u \in N(v)$ can affect $\bar \ell$ by at most the maximal value of $t_i(u,v,\gamma)$ realized by colour $\gamma$. By (\ref{betadef}) this is at most $L_i^{1/2-\beta}$. Changing the colour assignment to any other vertex cannot affect $\bar \ell$ at all. Changing the result of any equalizing flip can affect $\bar \ell$  by at most 1. \\
    \indent Suppose $\bar \ell \geq s$. We consider the two possible ways a colour $c$ could leave $v$'s list. A neighbour $u$ of $v$ could have caused a conflict, meaning that $u$ was assigned a colour $c'$ such that $(c,c') \in \Tau(v,u)$. The other way a colour could be removed is via an equalizing flip. Therefore, there must be $s_1$ neighbours of $v$ which were involved in conflicts and caused $v$ to lose $s_1$ distinct colours. There must also be $s_2$ equalizing flips which resulted in $s_2$ distinct colours (disjoint from the $s_1$ previous colours) being removed from $v's$ list, and we must have $s_1 + s_2=s$. For the first set of events we require two trials for each colour removed: the activation and colour assigned to the neighbour. For the second set of events we require only one trial, the equalizing flip. Therefore in total we require at most $2s_1 + s_2 \leq 2s$ trials to certify that $\bar \ell \geq s$. \\
    \indent Let $t = L_i^{1-\beta/2}$, and note that $\mathbf{E}(\bar \ell)$ = $L_i - \mathbf{E}(\ell_{i+1}(v)) = L_i - L_i \times \mathrm{Keep}_i$. So this expectation is at most $L_i$. Therefore since $L_i$ grows with $\Delta$, we have $t/2 \geq 20q\sqrt{\smash[b]{r\mathbf{E}(\bar \ell)}} + 64q^2r$ for any $q \leq L_i^{1/2-\beta}$ and any constant $r$, and for $\Delta$ sufficiently large. Using this fact, we have that 
    $$\mathbf{Pr}\Big(|\bar \ell - \mathbf{E}(\bar \ell)| > t\Big) \leq \mathbf{Pr}\Big(|\bar \ell - \mathbf{E}(\bar \ell)| > \frac{t}{2} + 20q\sqrt{r\mathbf{E}(\bar \ell)} + 64q^2r\Big).$$ So applying Talagrand's Inequality with $q=L_i^{1/2-\beta}$ and $r=2$, we get that
    \begin{align}
        \mathbf{Pr}\Big(|\bar \ell - \mathbf{E}(\bar \ell)| > L_i^{1-\beta/2}\Big) &\leq 4\exp\Bigg({-\frac{L_i^{2-\beta}}{64q^2\Big(\mathbf{E}(\bar \ell) + L_i^{1-\beta/2}\Big)}}\Bigg) \nonumber \\
        &\leq 4\exp\Bigg({-\frac{L_i^{1-\beta}}{128L_i^{1-2\beta}}}\Bigg) \nonumber \\
        &\leq 4\exp\Bigg({-\frac{L_i^{\beta}}{128}}\Bigg) \nonumber \\
        &\leq 4\exp\bigg({-\frac{\ln^3\Delta}{128}}\bigg) \nonumber \tag*{since $L_i \geq \Delta^{\frac{1}{2}\epsilon^2/(\epsilon+4)^2}$ by assumption} \\
        &< \Delta^{-\ln\Delta} \nonumber
    \end{align}

    \noindent for $\Delta$ sufficiently large. \\
    \indent Now by linearity of expectation, $\mathbf{E}(\ell_{i+1}(v)) = L_i - \mathbf{E}(\bar \ell)$. So we have that 
    $$\mathbf{Pr}\Big(|\ell_{i+1}(v) - \mathbf{E}(\ell_{i+1}(v))| > L_i^{1-\beta/2}\Big) = 
    \mathbf{Pr}\Big(|\bar \ell - \mathbf{E}(\bar \ell)| > L_i^{1-\beta/2}\Big) < \Delta^{-\ln\Delta}$$
    
    \noindent as desired.
    \item[$b)$]
    Let $X$ be the random variable which counts the number of constraints $(c,c') \in N_i(v,c)$ with endpoints $v,u$ such that $u$ did not retain a colour during iteration $i+1$. Let $Y$ be the number of constraints $(c, c') \in N_i(v,c)$ with endpoints $v,u$ such that during iteration $i+1$ the following occurred: 
    
    \begin{itemize}
        \item $u$ did not retain a colour, and,
        \item $c'$ was removed from $u$'s list by either an equalizing flip or because a neighbour of $u$ other than $v$ caused a conflict
    \end{itemize}

    Note that $t'_{i+1}(v,c) = X - Y$, and so by linearity of expectation it suffices to show that $X$ and $Y$ are both sufficiently concentrated.
    
    First we focus on $X$. We prove concentration by splitting up $X$ into two related variables, $X_1$ and $X_2$. Let $X_1$ denote the number of constraints $(c, c') \in N_i(v,c)$ with endpoints $v,u$ such that during iteration $i+1$ \textit{one} of the following occurred:
    
    \begin{itemize}
        \item $u$ was not activated, or,
        \item $u$ did not retain a colour because of a conflict with a neighbour other than $v$
    \end{itemize}
    
    Let $X_2$ denote the number of constraints $(c,c') \in N_i(v,c)$ with endpoints $v,u$ such that $u$ was activated and lost its colour because of a conflict with $v$. Let us be clear about the definitions of $X_1$ and $X_2$. Note that for a constraint $(c, c')$ between $v$ and $u$ such that $u$ lost its colour because of a conflict with $v$, this constraint could still be counted by $X_1$ so long as there is some neighbour $w \neq v$ which also conflicted with $u$. Similarly, if $u$ lost its colour because of a conflict with a neighbour $w \neq v$, it could still be counted by $X_2$ so long as $v$ also conflicted with $u$. Now note that $X_1 \leq X \leq X_1 + X_2$.
    
    First we consider $X_1$. Consider the effect on $X_1$ by changing a random choice for a vertex. By triangle-freeness, there are no edges within $N(v)$, and so for a vertex $u \in N(v)$, changing a random choice for $u$ only affects whether or not $u$ remains uncoloured. Thus this affects $X_1$ by at most $t_i(v,u,c)$. By (\ref{observebound}), this is at most $T_i^{1/2-\beta}$. Because our graph has no 4-cycles, if $w$ is a vertex outside the neighbourhood of $v$ and $w \neq v$, then $w$ can have at most 1 neighbour in $N(v)$, say $u$. Therefore changing a random choice for $w$ can at worst affect whether or not $u$ remains uncoloured, and therefore can affect $X_1$ again by at most $t_i(v,u,c) \leq T_i^{1/2-\beta}$. Of course, changing a random choice for $v$ has no effect on $X_1$.
    
    If vertices $v$ and $u$ are constrained by $(c,c') \in N_i(v,c)$ and $u$ does not receive a colour, the activation choice for $u$ certifies this fact. Otherwise if $u$ becomes uncoloured because of a conflict with a neighbour $u'$, then the activation and colour choices for both $u$ and $u'$ certify this fact. Thus we require a set of at most $4s$ random choices to certify that $X_1 \geq s$.
    
    As $X_1 \leq T_i$, we have $\mathbf{E}(X_1) \leq T_i$. So similar to the computation in part ($a$) which made use of Talagrand's Inequality, this implies that $$\mathbf{Pr}\big(|X_1 - \mathbf{E}(X_1)| > \frac{1}{4}T_i^{1-\beta/2}\big) < \frac{1}{4}\Delta^{-\ln\Delta}.$$
    
    Now we turn our attention to $X_2$. It's not hard to verify that $X_2$ is bounded above in distribution by the following random variable:  $$t_i(v,u^{\ast},\gamma)BIN\Big(T_i, \frac{K}{\ln\Delta} \times \frac{t_i(u^{\ast \ast},v,\gamma')}{L_i}\Big),$$ where $u^{\ast}, \gamma$ and $u^{\ast \ast}, \gamma'$ are chosen to maximize the values $t_i(v,u^{\ast},\gamma)$ and $t_i(u^{\ast \ast}, v, \gamma')$ respectively (recall that by (\ref{observebound}), both of these values are at most $T_i^{1/2-\beta}$).
    Since $T_i/L_i < \ln\Delta$ by \textcolor{defaultcolor}{Lemma \ref{decreasinglem}}, we have: $$\frac{K}{\ln\Delta} \times T_i \times \frac{t_i(u^{\ast \ast},v,\gamma')}{L_i} < Kt_i(u^{\ast \ast}, v, \gamma) < \alpha$$ for some $\alpha = O(T_i^{1/2-\beta})$. Rearranging terms, we get: $$\frac{K}{\ln\Delta} \times \frac{t_i(u^{\ast \ast},v,\gamma')}{L_i} < \frac{\alpha}{T_i}.$$ Since $T_i \geq \Delta^{\frac{1}{2}\epsilon^2/(\epsilon + 4)^2}$ by assumption, we can see that the Chernoff Bound implies:
    \begin{align}
        \mathbf{Pr}\big(X_2 > \frac{1}{4}T_i^{1-\beta/2}\big) &< \mathbf{Pr}\big(|BIN(T_i, \frac{\alpha}{T_i}) - \alpha| >  \frac{1}{4}\frac{T_i^{1-\beta/2}}{T_i^{1/2-\beta}} - \alpha\big) \nonumber \\
        &< \mathbf{Pr}\big(|BIN(T_i, \frac{\alpha}{T_i}) - \alpha| > T_i^{1/2}\big) \nonumber \\
        &< 2e^{-\Big(\big(1 + \frac{T_i^{1/2}}{\alpha}\big)\ln\big(1 + \frac{T_i^{1/2}}{\alpha}\big)-\frac{T_i^{1/2}}{\alpha}\Big)\alpha} \nonumber \\
        &= 2e^{-\Big(\ln\big(1 + \frac{T_i^{1/2}}{\alpha}\big) + \frac{T_i^{1/2}}{\alpha}\ln\big(1 + \frac{T_i^{1/2}}{\alpha}\big)-\frac{T_i^{1/2}}{\alpha}\Big)\alpha} \nonumber \\
        &< 2e^{-\Big(\frac{T_i^{1/2}}{\alpha}\big(\ln\big(1 + \frac{T_i^{1/2}}{\alpha}\big)-1\big)\Big)\alpha} \nonumber \\
        &< 2e^{-\Big(\frac{T_i^{1/2}}{\alpha}\big(\ln\big(1 + \Omega(T_i^{\beta})\big)-1\big)\Big)\alpha} \nonumber \\
        &< 2e^{-T_i^{1/2}} \nonumber \tag*{since $\ln(1 + \Omega(T_i^{\beta})) \gg 1$} \\
        &< 2e^{-\ln^3\Delta} \nonumber \\
        &< \frac{1}{4}\Delta^{-\ln\Delta} \nonumber
    \end{align}
    for $\Delta$ sufficiently large.
    
    Since $X_1 \leq X \leq X_1 + X_2$, these two bounds together imply that
     $$ \mathbf{Pr}\Big(|X - \mathbf{E}(X)| \geq \frac{1}{2}T_i^{1-\beta/2}\Big) < \frac{1}{2}\Delta^{-\ln\Delta}$$
    as desired.
    
    Now we focus on $Y$. Analogously define $Y_1$ to be the number of constraints $(c,c') \in N_i(v,c)$ with endpoints $v,u$ such that during iteration $i+1$ \textit{one} of the following occurred:
    
    \begin{itemize}
        \item $u$ was not activated, or,
        \item $u$ did not retain a colour because of a conflict with a neighbour other than $v$
    \end{itemize}
    
    \noindent and in addition:
    
    \begin{itemize}
        \item $c'$ was removed from $u$'s list because of an equalizing flip or because a neighbour of $u$ other than $v$ caused a conflict
    \end{itemize}
    
    Define $Y_2$ to be the number of constraints $(c,c') \in N_i(v,c)$ with endpoints $v,u$ such that during iteration $i+1$ the following occurred:
    
    \begin{itemize}
        \item $u$ did not retain a colour because of a conflict with $v$, and,
        \item $c'$ was removed from $u$'s list because of an equalizing flip or because a neighbour of $u$ other than $v$ caused a conflict
    \end{itemize}
    
    First we focus on $Y_1$, the same argument used for $X_1$ shows that changing any activation or colour assignment can affect $Y_1$ by at most $T_i^{1/2-\beta}$, and changing the result of an equalizing coin flip may have an affect of at most 1.
    
    If vertices $v$ and $u$ are constrained by $(c, c') \in N_i(v,c)$ and $u$ is not activated then the activation choice for $u$ certifies this fact. Otherwise if $u$ is uncoloured because of a conflict with a neighbour $w \neq v$, then the activation and colour choices for $u$ and $w$ certify this. In addition, if $c'$ is removed from $u$'s list because of a conflict with a neighbour other than $v$ then the activation and colour choice for that neighbour certifies this. So in total we require a set of at most $6s$ random choices to certify that $Y_1 \geq s$.
    
    Therefore applying Talagrand's Inequality, the same calculation from before gives $$\mathbf{Pr}\big(|Y_1 - \mathbf{E}(Y_1)| > \frac{1}{4}T_i^{1-\beta/2}\big) < \frac{1}{4}\Delta^{-\ln\Delta}.$$
    
    Now $Y_2$ is again bounded above by the random variable $T_i^{1/2-\beta}BIN\big(T_i, \frac{K}{\ln\Delta} \times \frac{T_i^{1/2-\beta}}{L_i}\big)$. So we have $$\mathbf{Pr}\big(Y_2 > \frac{1}{4}T_i^{1-\beta/2}\big) < \frac{1}{4}\Delta^{-\ln\Delta},$$ and this implies $Y$ is sufficiently concentrated as before, so the result follows.
\end{enumerate}
\end{proof}

Our next step is to show property $P(i)$ holds for all appropriate values of $i$.


\begin{lem}\label{propertyholdslem}
With positive probability, $P(i)$ holds for every $i$ such that for all $j < i: L_j, T_j \geq \Delta^{\frac{1}{2}\epsilon^2/(\epsilon + 4)^2}$ and $T_j \geq \frac{1}{8}L_j$. 
\end{lem}

\begin{proof}
The proof is by induction. For the base case, we note that the bounds $\ell_0(v) \geq L_0$ and $t_0(v,c) \leq T_0$ clearly hold by \textcolor{defaultcolor}{Modification \ref{mainmod}}. So $P(0)$ holds. \\
\indent For the induction step we assume $P(i)$ holds for some $i \in \mathbb{N}$, and we will show that with positive probability $P(i+1)$ holds.\\
\indent For every $v$, neighbour $u$ of $v$, and $c \in L_{i+1}(v)$ we define $A_v$ to be the event that $\ell_{i+1}(v) < L_{i+1}$, and $B_{v,c}$ be the event that $t_{i+1}(v,c) > T_{i+1}$. If none of these events hold, then $P(i+1)$ holds. \\
\indent If the event $A_v$ occurs, by \textcolor{defaultcolor}{Lemma \ref{expectationlem}} this implies that $|\ell_{i+1}(v) - \mathbf{E}(\ell_{i+1}(v))| > L_i^{1-\beta/2}$, and by \textcolor{defaultcolor}{Lemma \ref{concentrationlem}} this happens with probability at most $p=\Delta^{-\ln\Delta}$ (notice that we've implicitly used the inductive hypothesis here when making use of \textcolor{defaultcolor}{Lemma \ref{expectationlem} and Lemma \ref{concentrationlem}}).

Also if the event $B_{v,c}$ occurs we have:
\begin{align}
t_{i+1}(v,c) &> T_i(1-\frac{K}{\ln\Delta}\mathrm{Keep}_i) \times \mathrm{Keep}_i + T_i^{1-\beta/2} \tag*{by definition (\ref{eqn:eqs1})} \nonumber \\
&\geq t_i(v,c)\Big(1-\frac{K}{\ln\Delta}\mathrm{Keep}_i\Big) \times \mathrm{Keep}_i + T_i^{1-\beta/2} \tag*{by $IH$} \nonumber \\
&\geq t_i(v,c)\Big(1-\frac{K}{\ln\Delta}\mathrm{Keep}_i\Big) \times \mathrm{Keep}_i + T_i^{3/4} + \frac{1}{2}T_i^{1-\beta/2} \tag*{since $T_i \geq \Delta^{\frac{1}{2}\epsilon^2/(\epsilon + 4)^2}$ and recalling (\ref{betadef})} \nonumber \\
&\geq \mathbf{E}(t'_{i+1}(v,c)) + \frac{1}{2}T_i^{1-\beta/2} \tag*{by \textcolor{defaultcolor}{Lemma \ref{expectationlem}}} \nonumber
\end{align}

\noindent and since $t_{i+1}(v,c) \leq t'_{i+1}(v,c)$ we can see this implies $|t'_{i+1}(v,c) - \mathbf{E}(t'_{i+1}(v,c))| > \frac{1}{2}T_i^{1-\beta/2}$. Again by \textcolor{defaultcolor}{Lemma \ref{concentrationlem}}, this happens with probability at most $p$. \\
\indent Moreover, each event corresponding to a vertex $v$ is determined by equalizing coin flips and colours assigned to vertices of distance at most 2 from $v$. Therefore, by the Mutual Independence Principle (see chapter 4 of \cite{graph-colouring-book}), each event corresponding to a vertex $v$ is mutually independent of all events except for those corresponding to vertices of distance at most 4 from $v$. So every event can only be dependent on at most $d= \Delta^4 \times O\big(\sqrt{\Delta/\ln\Delta}\big) < \Delta^5$ other events. So $pd < \Delta^{-\ln\Delta}\Delta^5 < 1/4$ for $\Delta$ sufficiently large and the result follows from the Lov{\'a}sz Local Lemma.
\end{proof}

Our next lemma shows that the values for $L_{i+1}$ and $T_{i+1}$ defined in (\ref{eqn:eqs1}) do not stray too far from the values $L_{i+1}'$ and $T_{i+1}'$ defined in (\ref{eqn:eqs2}). As the recursive equations defined in (\ref{eqn:eqs2}) have a much nicer structure, this will simplify our analysis significantly moving forward.

\begin{lem}\label{closeequationslem}
If for all $j < i$ we have $L_j,T_j \geq \Delta^{\frac{1}{2}\epsilon^2/(\epsilon + 4)^2}$ and $T_j \geq \frac{1}{8}L_j$, then

\begin{enumerate}
    \item[$a)$] $|L_i - L'_i| \leq (L'_i)^{1-\beta/4} = o(L'_i)$;
    \item[$b)$] $|T_i - T'_i| \leq (T'_i)^{1-\beta/4} = o(T'_i)$.
\end{enumerate}
\end{lem}

\begin{proof}
\
\begin{enumerate}[wide=\parindent]
    \item[$a)$] The proof is by induction. Since $L'_i > L_i$, it suffices to show that $L_i' \leq L_i + (L'_i)^{1-\beta/4}$. Initially this is clearly the case. Now suppose that for some $i > 0$ we have $L_i' \leq L_i + (L'_i)^{1-\beta/4}$. It is easy to verify that the function $f(x) = x^{1-\beta/4} - x$ has only one critical point in the interval $(0, 1)$, namely at $x = (1-\beta/4)^{4/\beta} < e^{-1}$. Therefore $f(x)$ is decreasing on the interval $[e^{-1}, 1]$. By \textcolor{defaultcolor}{Lemma \ref{keepboundlem}} and \textcolor{defaultcolor}{Remark \ref{Kremark}}, $\mathrm{Keep}_i$ is within this interval. Also by \textcolor{defaultcolor}{Lemma \ref{keepboundlem}} we have $\mathrm{Keep}_i \leq (1-\textcolor{defaultcolor}{\frac{K}{10}\ln\Delta})$ and so using the fact that $f$ is decreasing we have 
    \begin{align}\label{keepbound}
        \mathrm{Keep}_i^{1-\beta/4} - \mathrm{Keep}_i \geq \Big(1-\textcolor{defaultcolor}{\frac{K}{10\ln\Delta}}\Big)^{1-\beta/4} - \Big(1-\textcolor{defaultcolor}{\frac{K}{10\ln\Delta}}\Big).
    \end{align}
    
    \noindent Now we use the Maclaurin Series expansion of the function $g(x) = (1-x)^{1-\beta/4} = 1 - (1-\beta/4)x - \frac{\beta}{8}(1-\beta/4)x^2 - \cdots$ to see that $(1-x)^{1-\beta/4} = 1-(1-\beta/4)x - O(x^2)$ for $x \in [0, 1]$. So for $x = \textcolor{defaultcolor}{\frac{K}{10\ln\Delta}}$ and for $\Delta$ sufficiently large, $$\Big(1-\textcolor{defaultcolor}{\frac{K}{10\ln\Delta}}\Big)^{1-\beta/4} \geq 1 - \textcolor{defaultcolor}{(1 - \beta/4)\frac{K}{10\ln\Delta}} - O\Big(\frac{1}{\ln^2\Delta}\Big) \geq 1 - \textcolor{defaultcolor}{(1-\beta/8)\frac{K}{10\ln\Delta}}$$ 
    So continuing (\ref{keepbound}) we see that 
    \begin{align}\label{keepbound2}
        \mathrm{Keep}_i^{1-\beta/4} - \mathrm{Keep}_i \geq 1 - \textcolor{defaultcolor}{(1 - \beta/8)\frac{K}{10\ln\Delta}} - \Big(1- \textcolor{defaultcolor}{\frac{K}{10\ln\Delta}}\Big) = \textcolor{defaultcolor}{\frac{\beta K}{80\ln\Delta}}.
    \end{align}
    
    \noindent Now we have: 
    \begin{align}
        L'_{i+1} &= \mathrm{Keep}_iL'_i \nonumber \tag*{by definition (\ref{eqn:eqs2})} \\
        &\leq \mathrm{Keep}_i(L_i + (L'_i)^{1-\beta/4}) \nonumber \tag*{by $IH$}\\
        & = \mathrm{Keep}_iL_i + \mathrm{Keep}_i(L'_i)^{1-\beta/4}\nonumber \\
        &= L_{i+1} + L_i^{1-\beta/2} + \mathrm{Keep}_i(L'_i)^{1-\beta/4} \nonumber \tag*{using definition (\ref{eqn:eqs1})} \\
        &\leq L_{i+1} + L^{1-\beta/2}_i + \mathrm{Keep}_i^{1-\beta/4}(L'_i)^{1-\beta/4} - \textcolor{defaultcolor}{\frac{\beta K}{80\ln\Delta}}(L'_i)^{1-\beta/4} \nonumber \tag*{by (\ref{keepbound2})} \\
        &= L_{i+1} + \mathrm{Keep}_i^{1-\beta/4}(L'_i)^{1-\beta/4} + L^{1-\beta/2}_i - \textcolor{defaultcolor}{\frac{\beta K}{80\ln\Delta}}(L'_i)^{1-\beta/4} \nonumber \\
        &\leq L_{i+1} + (L'_{i+1})^{1-\beta/4} \nonumber \tag*{using definition (\ref{eqn:eqs2})}
    \end{align}
    
    \noindent where in the last bound we use the fact that the negative term dominates over $L^{1-\beta/2}_i$ for $\Delta$ sufficiently large (since $L_i \geq \Delta^{\frac{1}{2}\epsilon^2/(\epsilon + 4)^2}$).
    
    \item[$b)$] For this part we have $T'_i < T_i$, and so it suffices to show that $T'_i \geq T_i - (T'_i)^{1-\beta/4}$. The proof uses the same method as before, except this time we show that for $x = \mathrm{Keep}_i(1-\frac{K}{\ln\Delta}\mathrm{Keep}_i)$ we have $x^{1-\beta/4} - x = \Omega\big(\frac{1}{\ln\Delta}\big)$, which allows us to complete the proof in the same way.
\end{enumerate}
\end{proof}

With the help of \textcolor{defaultcolor}{Lemma $\ref{closeequationslem}$} we can finally show that our list sizes never become too small, while on the other hand showing that $T_i$ eventually becomes much smaller than $L_i$. This will allow us to apply \textcolor{defaultcolor}{Lemma \ref{reedlem}}.


\begin{lem}\label{listsizelargelem}
There exists $i^{\ast}$ such that 
\begin{enumerate}
    \item[$a)$] For all $i < i^{\ast}, T_i, L_i > \Delta^{\frac{1}{2}\epsilon^2/(\epsilon + 4)^2}$, and $T_i \geq \frac{1}{8}L_i$;
    \item[$b)$] $T_{i^{\ast}} < \frac{1}{8}L_{i^{\ast}}$.
\end{enumerate}
\end{lem}

\begin{proof}
\
\begin{enumerate}[wide=\parindent]
    \item[$a)$] For this part we show that $T_i$ and $L_i$ are sufficiently large whenever $T_i \geq \frac{1}{8}L_i$. The proof is by induction, $L_0$ and $T_0$ are clearly sufficiently large. We show that for some $i$, if $T_i$ and $L_i$ are sufficiently large and $T_{i+1} \geq \frac{1}{8}L_{i+1}$, then $T_{i+1}$ and $L_{i+1}$ are sufficiently large as well. If $T_{i+1} < \frac{1}{8}L_{i+1}$ then there is no need to continue the procedure for more than $i+1$ iterations and this implies the existence of $i^{\ast}$. We show that this indeed happens in part ($b$). \\
    \indent First, recall the definitions of $L'_{i+1}$ and $T'_{i+1}$. Initially, we have: 
    
    \begin{gather}\label{initialdef}
        \begin{align}
            L'_0 &= L_0 = (\sqrt{2} + \epsilon/2)\sqrt{\Delta/\ln\Delta} \nonumber
            \\
            \\
            T'_0 &= T_0 = \frac{1}{\sqrt{2} + \epsilon/2}\sqrt{\Delta\ln\Delta} \nonumber
        \end{align}
    \end{gather}
    and we recursively define:

    \begin{gather}\label{recursivedef}
        \begin{align}
            L'_{i+1} &= L'_i \times \mathrm{Keep}_i \nonumber
            \\
            \\
            T'_{i+1} &= T'_i\Big(1-\frac{K}{\ln\Delta}\mathrm{Keep}_i\Big) \times \mathrm{Keep}_i \nonumber
        \end{align}
    \end{gather}
    
    We focus on showing $L_{i+1}$ is large first, but rather than finding a lower bound for $L_{i+1}$, \textcolor{defaultcolor}{Lemma \ref{closeequationslem}} allows us to focus on $L'_{i+1}$ instead.  In particular, for each $j \leq i$ we focus on bounding $\mathrm{Keep}_j$ which is closely related to the ratio $r_j = T_j/L_j \approx r'_j = T'_j/L'_j$. \\
    \indent Recall that we gave the following lower bound on $\mathrm{Keep}_j$ in the proof of \textcolor{defaultcolor}{Lemma \ref{keepboundlem}}: 
    \begin{align}
        \mathrm{Keep}_j > \exp\bigg(-K\frac{(1+\frac{1}{20}\epsilon e^{-\epsilon})r_j}{\ln\Delta}\bigg) > \exp\bigg(-K\frac{(1+\frac{1}{20}\epsilon e^{-\epsilon})}{(\sqrt{2} + \epsilon/2)^2}\bigg) \label{keepboundweak}
    \end{align}
    Now by (\ref{recursivedef}) we have $$r'_i = r_0\prod_{j=1}^{i-1}\Bigg(1 - \frac{K}{\ln\Delta}\mathrm{Keep}_j\Bigg).$$
    So by (\ref{keepboundweak}), and since $r_0 = \frac{\ln\Delta}{(\sqrt{2} + \epsilon/2)^2}$ by (\ref{initialdef}) we get:
    \begin{align}
        r'_i &\leq r_0 \times \bigg(1 - \frac{K}{\ln\Delta}\exp\Big(-K\frac{(1+\frac{1}{20}\epsilon e^{-\epsilon})}{(\sqrt{2} + \epsilon/2)^2}\Big)\bigg)^{i-1} \nonumber \\
        &= \frac{\ln\Delta}{(\sqrt{2} + \epsilon/2)^2} \bigg(1 - \frac{K}{\ln\Delta}\exp\Big(-K\frac{(1+\frac{1}{20}\epsilon e^{-\epsilon})}{(\sqrt{2} + \epsilon/2)^2}\Big)\bigg)^{i-1} \label{rbound}
    \end{align} \\
    Moreover, by \textcolor{defaultcolor}{Lemma \ref{closeequationslem}} we have \begin{align}
        r_i = \frac{T_i}{L_i} \leq \frac{T'_i + (T'_i)^{1-\beta/4}}{L'_i - (L'_i)^{1-\beta/4}} = r'_i \times \frac{1 + \frac{1}{(T'_i)^{\beta/4}}}{1 - \frac{1}{(L'_i)^{\beta/4}}} = (1+o(1))r'_i \label{strongrbound}
    \end{align}
    where in the last equality we use the fact that $L'_i, T'_i$ are close to $L_i,T_i$ respectively by \textcolor{defaultcolor}{Lemma \ref{closeequationslem}}, and the latter values satisfy $L_i,T_i \geq \Delta^{\frac{1}{2}\epsilon^2/(\epsilon+4)^2}$ by $IH$. So by (\ref{strongrbound}) we get $r_i \leq (1 + \epsilon e^{-\epsilon}/30)r'_i$ for $\Delta$ sufficiently large. So now we have:
    \begin{align}
        r_i &\leq (1 + \epsilon e^{-\epsilon}/30)r'_i \nonumber \\
        &\leq \frac{(1+ \epsilon e^{-\epsilon}/30)\ln\Delta}{(\sqrt{2} + \epsilon/2)^2}\bigg(1 - \frac{K}{\ln\Delta}\exp\Big(-K\frac{(1+\frac{1}{20}\epsilon e^{-\epsilon})}{(\sqrt{2} + \epsilon/2)^2}\Big)\bigg)^{i-1} \tag*{by (\ref{rbound})}
    \end{align}
    
    \noindent
    which yields a better lower bound on $\mathrm{Keep}_i$: 
    \begin{align}
        \mathrm{Keep}_i &> \exp\Bigg(-K(1+\epsilon e^{-\epsilon}/20)r_i/\ln\Delta\Bigg) \nonumber \tag*{by (\ref{keepboundweak})} \\
        &\geq \exp\Bigg(-\frac{K(1+\epsilon e^{-\epsilon}/20)(1+\epsilon e^{-\epsilon}/30)\ln\Delta}{\ln\Delta(\sqrt{2} + \epsilon/2)^2}\Big(1 - \frac{K}{\ln\Delta}e^{-\frac{K(1+\epsilon e^{-\epsilon}/20)}{(\sqrt{2} + \epsilon/2)^2}}\Big)^{i-1}\Bigg) \nonumber \\
        &= \exp\Bigg(-\frac{K(1+\epsilon e^{-\epsilon}/20)(1+\epsilon e^{-\epsilon}/30)}{(\sqrt{2} + \epsilon/2)^2}\Big(1 - \frac{K}{\ln\Delta}e^{-\frac{K(1+\epsilon e^{-\epsilon}/20)}{(\sqrt{2} + \epsilon/2)^2}}\Big)^{i-1}\Bigg) \label{keepboundstrong}
    \end{align}
    
    \noindent
    It is straightforward to check that for any $\epsilon > 0$, we have $(1+\epsilon e^{-\epsilon}/20)(1+\epsilon e^{-\epsilon}/30) \leq (1+\epsilon e^{-\epsilon}/10)$ in the above expressions. So now we can proceed with our bound for $L'_i$. We have:
    \begin{align}
       L'_{i+1} &= L_0 \prod_{j=1}^i\mathrm{Keep}_j \nonumber \\
       &\geq (\sqrt{2} + \epsilon/2)\sqrt{\frac{\Delta}{\ln\Delta}}\exp\Bigg(-\frac{K(1+\epsilon e^{-\epsilon}/10)}{(\sqrt{2} + \epsilon/2)^2}\sum_{j=1}^i\Big(1 - \frac{K}{\ln\Delta}e^{-\frac{K(1+\epsilon e^{-\epsilon}/20)}{(\sqrt{2} + \epsilon/2)^2}}\Big)^{j-1}\Bigg) \nonumber \tag*{by (\ref{keepboundstrong})} \\
       &\geq (\sqrt{2} + \epsilon/2)\sqrt{\frac{\Delta}{\ln\Delta}}\exp\Bigg(-\frac{(1+\epsilon e^{-\epsilon}/10)\ln\Delta}{(\sqrt{2} + \epsilon/2)^2}e^{\frac{K(1+\epsilon e^{-\epsilon}/20)}{(\sqrt{2} + \epsilon/2)^2}}\Bigg) \nonumber \tag*{by the geometric series} \\
       &= (\sqrt{2} + \epsilon/2)\sqrt{\frac{\Delta}{\ln\Delta}} \times \Delta^{-\frac{(1+\epsilon e^{-\epsilon}/10)}{(\sqrt{2} + \epsilon/2)^2}\exp\Big(\frac{K(1+\epsilon e^{-\epsilon}/20)}{(\sqrt{2} + \epsilon/2)^2}\Big)} \label{monsterexpression}
    \end{align}

    Now it suffices to obtain a bound on the exponent of $\Delta$ in (\ref{monsterexpression}). We simply choose a value of $K$ which makes the exponent collapse to some value strictly greater than $-\frac{1}{2}$. We take 
    \begin{align}
        K=\frac{(\sqrt{2} + \epsilon/2)^2}{(1+\epsilon e^{-\epsilon}/20)}\ln{(1+\epsilon e^{-\epsilon}/10)}. \label{Kdef}
    \end{align} At this point the reader should verify that this definition is consistent with remark \ref{Kremark}. So we have:
    $$
       \frac{(1+\epsilon e^{-\epsilon}/10)}{(\sqrt{2} + \epsilon/2)^2}\exp\Bigg(\frac{K(1+\epsilon e^{-\epsilon}/20)}{(\sqrt{2} + \epsilon/2)^2}\Bigg) = \Bigg(\frac{1+\epsilon e^{-\epsilon}/10}{\sqrt{2} + \epsilon/2}\Bigg)^2
    $$
    \noindent
    It is straightforward (although tedious) to verify that $\frac{1}{2}-\Big(\frac{1+\epsilon e^{-\epsilon}/10}{\sqrt{2} + \epsilon/2}\Big)^2 > \frac{1}{2}\Big(\frac{\epsilon}{\epsilon+3}\Big)^2$. So continuing (\ref{monsterexpression}) we obtain:
    \begin{align}
        L'_{i+1} &\geq (\sqrt{2} + \epsilon/2)\Delta^{\frac{1}{2}\epsilon^2/(\epsilon + 3)^2}/\sqrt{\ln\Delta} \nonumber \\
        &\geq \Delta^{\frac{1}{2}\epsilon^2/(\epsilon + 3.5)^2} \nonumber
    \end{align}
    \noindent
    for $\Delta$ sufficiently large.
    
    \indent Therefore $L_{i+1} > L'_{i+1} - (L'_{i+1})^{1-\beta/4} > \Delta^{\frac{1}{2}\epsilon^2/(\epsilon + 3.75)^2}$ for $\Delta$ sufficiently large. Since $T_{i+1} \geq \frac{1}{8}L_{i+1}$ this means $T_{i+1} > \Delta^{\frac{1}{2}\epsilon^2/(\epsilon + 4)^2}$, so both $L_{i+1}$ and $T_{i+1}$ are at least $\Delta^{\frac{1}{2}\epsilon^2/(\epsilon + 4)^2}$ for $\Delta$ sufficiently large.
    \item[$b)$] Now we will show that the ratio $r_i = T_i/L_i$ eventually becomes smaller than $\frac{1}{8}$.\\
    \indent We know by \textcolor{defaultcolor}{Lemma \ref{keepboundlem}} that $\mathrm{Keep}_0 \geq c$ for some positive constant $c$. We also know by (\ref{strongrbound}) that $r_i \leq (1 + o(1))r'_i \leq 2r'_i$ for $\Delta$ sufficiently large. Now let $\hat{i} = \frac{2}{Kc}\ln\Delta\ln\ln\Delta$. Recalling the value of $r_0=\frac{\ln\Delta}{(\sqrt{2} + \epsilon/2)^2}$ we have: $$r_{\hat{i}} \leq 2r_0\Big(1 - \frac{Kc}{\ln\Delta}\Big)^{\frac{2}{Kc}\ln\Delta\ln\ln\Delta} \leq  \frac{2}{(\sqrt{2} + \epsilon/2)^2}\ln\Delta \times e^{-2\ln\ln\Delta} < \frac{1}{8}.$$ for sufficiently large $\Delta$.
    Let $S = \{i \leq \hat{i}: r_i < \frac{1}{8}\}$. Clearly $S$ is non-empty, so we choose $i^{\ast}$ to be the minimum value in $S$. Now the fact that $r_i$ is decreasing implies the lemma.
\end{enumerate}
\end{proof}

Finally, we put it all together to prove the main theorem.

\bigskip
\noindent
\textcolor{defaultcolor}{\textbf{Proof of Theorem \ref{conflict-theorem}.}}

Let $i^{\ast}$ be the natural number guaranteed by \textcolor{defaultcolor}{Lemma \ref{listsizelargelem}}. Firstly we can carry out the wasteful colouring procedure for $i^{\ast}$ iterations. By \textcolor{defaultcolor}{Lemma \ref{propertyholdslem} and Lemma \ref{listsizelargelem}}, $P(i)$ holds with positive probability for each $i \leq i^{\ast}$ so we do in fact perform $i^{\ast}$ iterations. Thus by property $P(i^{\ast})$ and \textcolor{defaultcolor}{Lemma \ref{listsizelargelem}}, we will have that for each $v$ and $c \in L(v)$, $t_{i^{\ast}}(v,c) \leq T_{i^{\ast}} \leq \frac{1}{8}L_{i^{\ast}}$, and the colouring can be completed by \textcolor{defaultcolor}{Lemma \ref{reedlem}}.

\section{Concluding Remarks}
In this paper we showed that for any $\epsilon > 0$, there exists $\Delta_0$ such that any multigraph with maximum degree $\Delta > \Delta_0$, no cycles of length 3 or 4 and an edge labelling $\tau$ with $D(\tau)$ at most $\Delta^{\frac{1}{4}\epsilon^2/(\epsilon + 5)^2}$ can be conflict coloured with at most $\ell=(2\sqrt{2} + \epsilon)\sqrt{\Delta/\ln\Delta}$ colours.

Note that for large $\epsilon$, our result means that a bound on $D(\tau)$ of $\Delta^{\frac{1}{4}- o(1)}$ is sufficient to complete the colouring with $\ell$ colours (here $o(1)$ is a function of $\epsilon$). On the other hand, \textcolor{defaultcolor}{Example \ref{counterexample2}} says that a bound of $\Delta^{\frac{1}{2}-o(1)}$ is necessary. Closing this gap could be an interesting problem for future work.

For small $\epsilon$ the gap is more significant since when $\epsilon$ tends to 0, theorem \ref{conflict-theorem} gives a bound which tends to $\Delta^0$ while \textcolor{defaultcolor}{Example \ref{counterexample2}} says that a bound of $\ell^{15/16} \approx \Delta^{15/32}$ is necessary. Closing this gap could be another interesting problem. A step in this direction is to answer the following questions:

\bigskip
\textbf{Question 1.} Is there a constant $d \in (0, \frac{15}{32})$ such that any multigraph with maximum degree $\Delta$, no cycles of length 3 or 4, and an edge labelling with conflict degree at most $\Delta^d$ can be conflict coloured with lists of size $(2\sqrt{2} + o(1))\sqrt{\Delta/\ln\Delta}$?

\bigskip
\textbf{Question 2.} If the answer to Question 1 is affirmative, what is the biggest value of $d$ possible?

\bigskip
We conclude by pointing out that these questions cannot be resolved in the affirmative with the analysis presented here. Recall \textcolor{defaultcolor}{Observation \ref{mainobserve}} and the following discussion. There we showed that if our list sizes are sufficiently large we have:
$$ t_i(v,u,c) \leq L_i^{1/2 -\beta}.$$
Recall that we required $t_i(v,u,c)$ to be quite small, and so this bound was important for proving the concentration results (\textcolor{defaultcolor}{Lemma \ref{concentrationlem}}) of our main parameters. 

Now when we allow the conflict degree to be as large as $\Delta^d$ for come absolute constant $d$, then initially $t_i(v,u,c)$ can also be this large. This becomes a significant problem if $t_i(v,u,c)$ remains this large throughout the procedure, since towards the end of the procedure our lists are of size $L_i \approx \Delta^{\epsilon^2}$. In particular, when $\epsilon$ is much smaller than $d$ it can be the case that $t_i(v,u,c)$ is as big as $L_i$. I.e. assigning $c$ to $v$ in this case could eliminate $u$'s list entirely.

A natural attempt to overcome this is to show that the parameter $t_i(v,u,c)$ decreases quickly with each iteration. Unfortunately, there are cases where this does not happen with high probability.

Nevertheless, as far as the authors know Questions 1 and 2 could conceivably be answered in the affirmative via other methods.

\section{Acknowledgements}

The authors would like to thank Aleksandar Nikolov for his careful review and valuable comments. This research is supported by an NSERC Discovery Grant.

\bibliographystyle{siam}
\bibliography{references}

\end{document}